\pdfoutput=1
\documentclass[intlimits,a4paper,10pt,reqno]{elsarticle-mr}

\usepackage[%
   colorlinks=true,
   urlcolor=rltblue,       
   filecolor=rltgreen,     
   citecolor=rltgreen,     
   linkcolor=rltred,       
   bookmarks=true,
   unicode,
   plainpages=false,
   ]{hyperref}

\usepackage[operator,rose,frak]{paper_diening}
\usepackage[active]{srcltx}
\usepackage{amsmath,color}
\usepackage{amssymb,amsthm,amsxtra,amscd}
\usepackage{microtype}
\usepackage{esint}
\usepackage[applemac]{inputenc}

\definecolor{rltred}{rgb}{0.75,0,0}
\definecolor{rltgreen}{rgb}{0,0.5,0}
\definecolor{rltblue}{rgb}{0,0,0.75}

\newtheorem{Def}[equation]{Definition}
\newtheorem{Sa}[equation]{Theorem}
\newtheorem{Lem}[equation]{Lemma}

\newtheorem{Bem}[equation]{Remark}
\newtheorem{Kor}[equation]{Corollary}
\newtheorem{Vss}[equation]{Assumption}

\newenvironment{Bew}{\begin{proof}[Proof]}{\end{proof}}

\newcommand{\R}{\mathbb{R}}
\newcommand{\E}{\mathbf{E}}
\newcommand{\D}{\mathbf{D}}
\newcommand{\Rr}{\mathbf{R}}
\newcommand{\w}{\boldsymbol\omega}

\newcommand{\x}{\boldsymbol\xi}

\def\dint{\fint}

\DeclareMathOperator{\Div}{\operatorname{div}}

\newcommand{\Ss}{\mathbf{S}\big(\mathbf{D}\mathbf{v},\mathbf{R}(\mathbf{v},\boldsymbol\omega),\mathbf{E}\big)}

\newcommand{\Sn}{\mathbf{S}\big(\mathbf{D}\mathbf{v}^n,\mathbf{R}(\mathbf{v}^n,\boldsymbol\omega^n),\mathbf{E}\big)}

\newcommand{\N}{\mathbf{N}(\nabla\boldsymbol\omega,\mathbf{E})}

\newcommand{\Nn}{\mathbf{N}(\nabla\boldsymbol\omega^n,\mathbf{E})}

\newcommand{\anti}{{\ensuremath{\mathrm{skew}}}}
\newcommand{\bell}{\boldsymbol{\ell}}

\newcommand{\Lp}{L^p(\Omega)}

\newcommand{\Lq}{L^q(\Omega)}
\newcommand{\Lr}{L^r(\Omega)}

\newcommand{\Wpo}{W^{1,p}_0(\Omega)}

\newcommand{\WpEo}{H^{1,p}_0(\Omega;\vert \mathbf{E}\vert^2)}

\newcommand{\Wpwo}{H^{1,p}_0(\Omega;\sigma)}
\newcommand{\Wpw}{W^{1,p}(\Omega;\sigma)}
\newcommand{\Hpw}{H^{1,p}(\Omega;\sigma)}

\newcommand{\Lpw}{L^{p}(\Omega;\sigma)}

\newcommand{\Vp}{V_p(\Omega)}

\newcommand{\Wpwx}{W^{1,p(\cdot)}(\Omega;\sigma)}
\newcommand{\Hpwx}{H^{1,p(\cdot)}(\Omega;\sigma)}
\newcommand{\Lpwx}{L^{p(\cdot)}(\Omega;\sigma)}
\newcommand{\Io}{\int_{\Omega}}

\numberwithin{equation}{section}

\begin{document}
\begin{frontmatter}

\title{Existence of steady solutions for a general model for micropolar electrorheological fluid flows}

\author[ak]{Alex Kaltenbach}
\ead{alex.kaltenbach@mathematik.uni-freiburg.de}

\author[mr]{\corref{cor1}Michael R\r u\v zi\v cka}
\ead{rose@mathematik.uni-freiburg.de}

\cortext[cor1]{Corresponding author}

\address[mr]{Institute of Applied Mathematics,
  Albert--Ludwigs--University Freiburg, Ernst--Zermelo--Str.~1, D-79104 Freiburg,
  GERMANY.}  

\address[ak]{Institute of Applied Mathematics,
	Albert--Ludwigs--University Freiburg, Ernst--Zermelo--Str.~1, D-79104 Freiburg,
	GERMANY.}

\begin{abstract}
  In this paper we study the existence of solutions to a steady
  system that describes the motion of a micropolar electrorheological
  fluid. The constitutive relations for the stress tensors belong to
  the class of generalized Newtonian fluids.  The analysis of this
  particular problem leads naturally to weighted~Sobolev~spaces.  By
  deploying the Lipschitz truncation technique, we establish the
  existence~of solutions without additional assumptions on the
  electric field.
\end{abstract}

\begin{keyword}
Existence of solutions, Lipschitz truncation, weighted function spaces, variable exponent spaces, micropolar
  electrorheological fluids.

  \MSC 35Q35  
  \sep 35J92  
  \sep
  46E35  
\end{keyword}

\end{frontmatter}

\enlargethispage{3mm}
\vspace*{-3mm}
\section{Introduction}\label{introduction}
In this paper we establish the existence of solutions of the
system\footnote{We denote by $\bfvarepsilon$ the isotropic third order
	tensor and by $\bfvarepsilon:\mathbf{S}$ the vector with the
	components $\vep_{ijk}S_{jk}$, $i=1,\ldots,d$, where the summation
	convention over repeated indices is used.}\\[-5mm]
\begin{align}
	\begin{aligned}\label{NS}
		-\Div \mathbf{S}+\Div(\mathbf{v}\otimes\mathbf{v})
		+\nabla\pi&=\mathbf{f} &&\text{in} \ \Omega\,,
		\\
		\Div \mathbf{v}&=0 &&\text{in} \ \Omega\,,
		\\
		-\Div\mathbf{N}+\Div(\boldsymbol{\omega}\otimes
		\mathbf{v})&=\boldsymbol{\ell}-\bfvarepsilon:\mathbf{S}
		&\quad&\text{in} \ \Omega\,,\\
		\mathbf{v}=\bfzero\,, \quad \boldsymbol\omega &=\bfzero
		&&\text{on}\ \partial\Omega \,. \\[-2mm]
	\end{aligned}
\end{align}
Here, $\Omega\subseteq \setR^d$, $d\ge 2$, is a bounded domain. The three
equations in \eqref{NS} represent the balance of momentum, mass and angular
momentum for an incompressible, micropolar electrorheological
fluid.  In it, $\bv$ denotes the velocity, $\w$ the
micro-rotation, $\pi$ the pressure, $\bS$ the mechanical extra
stress tensor, $\bN$ the couple stress tensor, $\bell$ the
electromagnetic couple force, ${\ff=\tilde \ff + \chi^E\divo (\bE
	\otimes \bE)}$ the body force, where $\tilde \ff$ is the mechanical
body force, $\chi^E$ the dielectric susceptibility and $\bE$ the
electric field. The electric field $\bE$ solves the quasi-static
Maxwell's equations \\[-4mm]
\begin{align}
	\begin{aligned}\label{maxwell}
		&& \Div \mathbf{E}&=0 &&\text{in}\ \Omega\,,
		\\
		&& \curl \mathbf{E}&=\bfzero &&\text{in}\ \Omega\,,
		\\
		&& \mathbf{E}\cdot \mathbf{n}&=\mathbf{E}_0\cdot \mathbf{n}
		&\quad&\text{on}\ \partial\Omega\,, \\[-5mm]
	\end{aligned}
\end{align}
where $\mathbf{n}$ is the outer normal vector field of
$\partial \Omega$ and $\mathbf{E}_0$ is a given~electric~field. The
system \eqref{NS}, \eqref{maxwell} is the steady version of a model
derived in \cite{win-r}, which generalizes previous models of
electrorheological fluids in \cite{RR2},
\cite{rubo}.~The~model~in~\cite{win-r} contains a more realistic
description of the dependence of the electrorheological effect on the
direction of the electric field. Since Maxwell's equations
\eqref{maxwell} are separated from the balance laws \eqref{NS} and due
to the well developed \mbox{mathematical} theory for Maxwell's equations (cf.~Section~\ref{sec:E}), we can
view the electric~field~$\bE$ with appropriate properties as a given
quantity in \eqref{NS}. As a consequence,~we concentrate in this paper on the investigation of the mechanical properties~of~the
electrorheological fluid governed by \eqref{NS}.

A representative example for a constitutive
relation for the stress tensors in \eqref{NS} reads, e.g., (cf.~\cite{win-r},
\cite{rubo})
\begin{align}
	\hspace*{-1mm}  
	\begin{aligned}\label{eq:SN-ex}
		\mathbf{S}&=(\alpha_{31}+\alpha_{33}\vert\mathbf{E}\vert^2)
		(1+\vert\mathbf{D}\vert)^{p-2}\mathbf{D}+ \alpha_{51}
		(1+\vert\mathbf{D}\vert)^{p-2}\big (\mathbf{D} \bE \otimes \bE +
		\bE \otimes \bD\bE \big)\hspace*{-5mm}
		\\
		&\quad +
		\alpha_{71}\vert\mathbf{E}\vert^2(1+\vert\mathbf{R}\vert)^{p-2}
		\mathbf{R} + \alpha_{91} (1+\vert\mathbf{R}\vert)^{p-2}\big
		(\mathbf{R} \bE \otimes \bE + \bE \otimes \bR\bE \big)\,,
		\\
		\mathbf{N}&=(\beta_{31}+\beta_{33}\vert\mathbf{E}\vert^2)
		(1+\vert\nabla\boldsymbol\omega\vert)^{p-2}\nabla\boldsymbol\omega
		\\
		&\quad + \beta_{51}(1+\vert\nabla \w\vert)^{p-2}\big ((\nabla \w
		)\bE \otimes \bE + \bE \otimes (\nabla \w)\bE \big)\,,
	\end{aligned}
\end{align}
with material constants $\alpha_{31},\alpha_{33},\alpha_{71},\beta_{33}>
0$ and $\beta_{31}\ge 0$ and a shear exponent $p=\hat p \circ
\abs{\bE}^2$, where $\hat p$ is a material function. 
In \eqref{eq:SN-ex}, we employed the common notation\footnote{Here,
  $\bfepsilon:\bv$ denotes the tensor with components $\vep_{ijk}v_k$,
  $i,j=1,\ldots,d$.} $\bD = (\nabla \bv)^\sym$ and
${\bR=\bR(\bv,\w):= (\nabla \bv)^\anti +\bfvarepsilon :\w }$. 

Micropolar fluids have been introduced by Eringen in the sixties
(cf.~\cite{eringen-book}).~A model for electrorheological fluids was
proposed in \cite{RR1},~\cite{RR2},~\cite{rubo}.~While~there~exist
many investigations of micropolar fluids or
electrorheological~fluids (cf.~\cite{Lukaszewicz},~\cite{rubo}), there
exist to our knowledge no mathematical investigations of
steady~motions~of micropolar electrorheological fluids except the PhD
thesis \cite{frank-phd},~the~diploma~thesis \cite{weber-dipl} and the
research paper \cite{erw}. Even these investigations only treat the
case of constant shear exponents.

For the existence theory of problems of similar type as \eqref{NS},
the Lipschitz truncation technique (cf.~\cite{fms2}, \cite{dms}) has
proven to be very~powerful.~This~method is available in the setting of
Sobolev spaces (cf.~\cite{fms}, \cite{dms}, \cite{john}), variable
exponent Sobolev spaces (cf.~\cite{dms}, \cite{john}), solenoidal
Sobolev spaces (cf.~\cite{bdf}), Sobolev spaces with 
Muckenhoupt~weights~(cf.~\cite{erw}) and functions~of~bounded~\mbox{variation} (cf.~\cite{BDG19}).  Since, in general,
$\vert \bE\vert^2$ does not belong to~the~correct Muckenhoupt class $\mathcal A_p$,
the results in \cite[Thm.~5.49, Thm.~5.56, Thm.~5.59 \&
Thm.~6.44]{erw} are either sub-optimal with respect to the lower
bound for the shear~exponent~$p$ or require additional restrictive assumptions on
the electric field $\bE$.~Apart~from~that, solely the case of constant shear
exponents is treated. As a consequence, there are no results for the general model
for micropolar electrorheological fluids \eqref{NS}--\eqref{eq:SN-ex},
which is the most realistic from the point of view of modeling and
applications. The present paper improves the previous treatments in
two~special~aspects.~First, we show the existence of solutions for constant
shear exponents $p$ larger than the optimal exponent $\frac {2d}{d+2}$
without the restrictive assumption that  $\vert \bE\vert^2$
belongs to the Muckenhoupt class $\mathcal A_p$. Second, we extend
this result to the general case of shear exponents
$p=\hat p \circ \abs{\bE}^2$ satisfying $p^->\frac {2d}{d+2}$. In
fact, this seems to be the first existence result in weighted variable
exponent Sobolev spaces with a weight not satisfying a Muckenhoupt condition. 

\smallskip \textit{This paper is organized as follows:} First, we
introduce the functional~setting in the constant exponent case,
collect auxiliary results and give assumptions for the stress
tensors. Section~\ref{sec:E} is devoted to the analysis of the
electric field and weighted Sobolev spaces, while Section
\ref{sec:stab} is devoted to the weak stability of the stress
tensors. In Section \ref{veroeffentlichung3}, we deploy the Lipschitz
truncation technique in order to prove the existence of solutions of
\eqref{NS}, \eqref{maxwell} for constant~shear~exponents.~Section~\ref{veroeffentlichung1p(x)} contains the generalization of the
previous~results~to~the~variable~exponent~case. 

\section{Preliminaries}\label{veroeffentlichung1}

\subsection{Notation and function spaces}

We employ the customary Lebesgue spaces $\Lp$, $1\leq p\leq\infty$, and
Sobolev spaces $W^{1,p}(\Omega)$, $1\leq p\leq\infty$, where
$\Omega\subseteq \R^d$, $d\in \mathbb{N}$, is a bounded domain. We
denote by $\Vert\cdot\Vert_p$ the norm in $\Lp$ and by
$\Vert\cdot\Vert_{1,p}$ the norm~in~$W^{1,p}(\Omega)$.~Moreover, the
spaces $C_0^k(\Omega)$, $k \in \setN_0\cup \set{\infty}$, consist of
$k$--times continuously differentiable functions with compact support
in $\Omega$. The space $W^{1,p}_0(\Omega)$, $1\le p <\infty$, is
defined as the completion of $C_0^\infty(\Omega)$ with respect to the
gradient norm $\Vert\nabla \cdot\Vert_{p}$, while the space
$V_p(\Omega)$, $1\le p <\infty$, is the closure of
$C_{0,\textup{div}}^\infty(\Omega):=\{\bu\in C_0^\infty(\Omega)^d\fdg
\Div \bu=0\}$ with respect to the gradient norm
$\Vert\nabla
\cdot\Vert_{p}$.~For~a~bounded~Lipschitz~domain~${G\subseteq \R^d} $,
we define $W^{1,\infty}_0(G)$ as the subspace of functions
$u \in W^{1,\infty}(G)$ having a vanishing trace, i.e.,
$u|_{\partial G}=0$. We use small boldface letters, e.g.,~$\bv$, to
denote vector-valued functions and capital boldface letters,
e.g.,~$\bS$, to denote~\mbox{tensor-valued}~functions\footnote{The
  only exception of this is the electric vector field which is denoted
  as usual by $\bE$.}. However, we do not distinguish between scalar,
vector-valued and tensor-valued function spaces in the notation. The
standard scalar product between vectors is denoted by $\bv\cdot \bu$,
while the standard scalar product between tensors is
denoted~by~$\bA:\bB$. For a normed linear vector space $X$, we denote
its topological dual space by $X^*$.  Moreover, we employ the notation
$\langle u,v\rangle:=\Io uv\,dx$, whenever the right-hand side is
well-defined.  We denote by $\vert M\vert$ the $d$--dimensional
Lebesgue measure of a measurable set $M$. The mean value of a locally
integrable function $u\in L^1_{\loc}(\Omega)$ over a measurable set
$M\subseteq\Omega$ is denoted by
$\dint_M u\,dx:=\frac{1}{\vert M\vert} \int_M u\,dx$. By
$L^p_0(\Omega)$ and $C^\infty_{0,0}(\Omega)$, resp., we denote the
subspace of $L^p(\Omega)$ and $C^\infty_{0}(\Omega)$, resp.,
consisting of all functions $u$ with
vanishing~mean~value,~i.e.,~${\dint_\Omega u\,dx =0}$.

We will also use weighted Lebesgue and Sobolev spaces
(cf.~\cite{heinonen}, \cite{kufner_opic}, \cite{kfj}). A weight $\sigma$ on $\R^d$
is a locally integrable function satisfying $0<\sigma(x)<\infty$
a.e.\footnote{If not stated otherwise, a.e.~is meant with respect to
  the Lebesgue measure.}. To each weight $\sigma$ we associate a Radon
measure $\nu_\sigma$ defined via $\nu_\sigma (A):=\int _A \sigma\, dx$. The
space $\Lpw$, $p \in [1,\infty)$, is defined as the set of all Lebesgue
measurable functions $u:\Omega\to \mathbb{R}$ for which $\Io\vert
u\vert^p\sigma\,dx<\infty$. It is a Banach space if equipped with the
norm $\Vert u\Vert_{p,\sigma}\hspace*{-0.1em}:=\hspace*{-0.1em} \big(\hspace*{-0.1em}\Io\hspace*{-0.1em} \vert u\vert^p
\sigma\, dx\big)^{\smash{\frac{1}{p}}}\!$. \!For $p\hspace*{-0.12em} \in\hspace*{-0.12em} (1,\hspace*{-0.1em}\infty)$,~it~is~separable~and~reflexive. 
Note that, in general, the space $\Lpw$ does not embed into
$L^1_\loc(\Omega)$ (cf.~\cite{kufner_opic}). The condition $\smash{\sigma ^{\frac
  {-1}{p-1}} }\in L_\loc ^1(\Omega)$ is both necessary and sufficient for
the  embedding $\Lpw \vnor L^1_\loc(\Omega)$ (cf.~\cite{kufner_opic}, \cite{frank-phd}). 
The dual space of $\Lpw$ can be
identified with respect to $\skp{\cdot}{\cdot}$~with~$
\smash{L^{p'}(\Omega;\sigma')}$, where $\smash{\sigma':=\sigma^{\smash{\frac{-1}{p-1}}}}$. In
particular,~we~have~that
\begin{equation*}
  \abs{\skp{u}{v}}\le \norm{u}_{p,\sigma} \norm{v}_{p',\sigma'}\,,
\end{equation*}
if $\smash{u\hspace*{-0.1em} \in\hspace*{-0.1em} L^{p}(\Omega;\sigma)}$ and  $\smash{v \hspace*{-0.1em}\in\hspace*{-0.1em} L^{p'}(\Omega;\sigma')}$. \hspace*{-0.05em}By
$L^p_0(\Omega;\sigma)$, we denote the subspace~of~$L^p(\Omega;\sigma)$
consisting of all functions with vanishing mean value.

In order to define weighted Sobolev spaces, we make the following
assumption on the weight $\sigma$.
\begin{Vss}\label{weight}
  Let $\Omega\subseteq\R^d$, $d\in \mathbb{N}$, be an open set and
  $p\in \left[1,\infty\right)$.  The weight $\sigma$ is admissible,
  i.e., if a sequence ${(\varphi_n)_{n\in \mathbb{N}}\subseteq C^\infty(\Omega)}$ and
  ${\bv \in L^p(\Omega;\sigma)}$ satisfy
  ${\int_\Omega\vert\varphi_n\vert^{p}\sigma\,dx\to  0}$
  $(n\to \infty)$ and
  ${\int_\Omega\vert\nabla\varphi_n-\bv\vert^{p}\sigma\,dx\to
     0}$ $(n\to\infty)$, then it follows that $\bv=\mathbf{0} $ in
  $L^p(\Omega;\sigma)$.
\end{Vss}
\begin{Bem}\label{weightexamples}
  \begin{enumerate}
  \item[{\rm (i)}] If $\sigma$ belongs to the Muckenhoupt class
    $ \mathcal A_p$ for some ${p\hspace*{-0.15em}\in\hspace*{-0.15em} \left[1,\infty\right)}$, then
    Assumption~\ref{weight}~is~satisfied for this specific $p$
    (cf. \cite[Sec.~1.9]{heinonen}).
		
  \item[{\rm (ii)}] If $\sigma \in C^0(\Omega)$, then Assumption
    \ref{weight} is satisfied for all
    $p\in \left[1,\infty\right)$.~In~fact, the set
    $\Omega_{\sigma}:=\{\sigma>0\}$ is open and satisfies
    $\vert \Omega\setminus\Omega_{\sigma}\vert=0$. In addition, for
    any $K\subset\subset \Omega_{\sigma}$, there exists a constant
    $c_K>0$ such that~${c_K^{-1}\leq \sigma\leq c_K}$~in~$K$. Thus,
    for a sequence
    $(\varphi_n)_{n\in \mathbb{N}}\subseteq C^\infty(\Omega)$ from
     $\int_\Omega\vert\varphi_n\vert^p\,\sigma\,dx\to {0}$ $(n\to\infty)$ and
    ${\int_\Omega\vert\nabla\varphi_n-\bv\vert^p\,\sigma\,dx\to 0}$
    $(n\to\infty)$, where $\bv \in L^p(\Omega;\sigma)$, it follows
    that $\varphi_n\hspace*{-0.12em}\to\hspace*{-0.12em} 0$ in $L^p(K)$ $(n\hspace*{-0.12em}\to \hspace*{-0.12em}\infty)$ and
    $\nabla\varphi_n\hspace*{-0.12em}\to\hspace*{-0.12em} \bv$~in~$L^p(K)$~${(n\hspace*{-0.12em}\to\hspace*{-0.12em} \infty)}$~for~all~${K\!\!\subset\subset\! \Omega_{\sigma}}$. Consequently, for every
    $\boldsymbol\psi\in C_0^\infty(\Omega_{\sigma})$, one has that
    \begin{align*}
      0=\lim_{n\to \infty}{-\int_{\Omega}{\varphi_n\textup{div}\,
      \boldsymbol\psi\,dx}}=\lim_{n\to
      \infty}{\int_{\Omega}{\nabla\varphi_n\cdot
      \boldsymbol\psi\,dx}}=\int_{\Omega}{\bfv\cdot
      \boldsymbol\psi\,dx}\,, 
    \end{align*}
    i.e., $\bfv=0$ a.e.~in $\Omega_{\sigma}$, which, in turn, implies
    that $\bfv=0$ a.e.~in $\Omega$.
    \item[\rm (iii)] There exist weights $\sigma$ such that Assumption
      \ref{weight} is not satisfied (cf.~\cite{FKS82}). 
  \end{enumerate}
\end{Bem}

For $\sigma$ satisfying Assumption~\ref{weight}, 
and $p \in [1,\infty)$, we introduce the norm
$$
  \Vert u\Vert_{1,p,\sigma}:=
  	\|u\|_{p,\sigma}
  +\|\nabla u\|_{p,\sigma}\,,
$$ 
whenever the right-hand side is well-defined.  Then, the Sobolev space
$\Hpw$ is defined to be the completion of
$$
\mathcal{V}_{p,\sigma}:=\big\{u\in C^\infty(\Omega)\fdg
\Vert u\Vert_{1,p,\sigma}<\infty\big\}
$$
with respect to the norm $\Vert\cdot\Vert_{1,p,\sigma}$. In other
words, $u \in \Hpw$ if and only if $u\in \Lpw$ and there exists a function
$\bv\in\Lpw$ such that for some sequence
$(\varphi_n)_{n\in \mathbb{N}}\hspace*{-0.1em}\subseteq\hspace*{-0.1em} C^\infty(\Omega)$ holds
$\Io\vert\varphi_n-u\vert^p\sigma\,dx\hspace*{-0.1em}\to\hspace*{-0.1em} 0$ $(n\hspace*{-0.1em}\to\hspace*{-0.1em}\infty)$~and~${\Io\vert\nabla\varphi_n-\bv\vert^p\sigma\,dx\hspace*{-0.1em}\to\hspace*{-0.1em} 0}$
${(n\to\infty)}$. In~this~case, the function $\bv$ is called the
gradient of $u$ in $\Hpw$ and denoted by $\hat{\nabla}u:=\bv$. Here,
Assumption \ref{weight} implies that $\hat{\nabla}u $ is a uniquely
defined function in $\Lpw$. Note that $W^{1,p}(\Omega)=\Hpw$ if
$\sigma= 1$ a.e.~in $\Omega$ with
${\nabla u =\hat\nabla u}$~for~all~${u\in W^{1,p}(\Omega)}$. However,
in general, $\hat{\nabla}u$ and the usual weak or distributional gradient
$\nabla u$ do not coincide.~The~space~$\Hpw$,~${p \in (1,\infty)}$,
is a separable and reflexive Banach space. Then, we define the space $\Wpwo$~as~the completion of $C_0^\infty(\Omega)$ with respect to
$\Vert\cdot\Vert_{1,p,\sigma}$.  We will use the 
observation that, if $\sigma\in L^\infty(\Omega)$, then
$W^{1,p}_0(\Omega)\hookrightarrow H^{1,p}_0(\Omega;\sigma)$~and
${\nabla u =\hat\nabla u}$ for every ${u\in W^{1,p}_0(\Omega)}$ (cf.~\cite[Lem.~1.12]{heinonen}), which is a 
consequence of the
inequality
$\smash{\|u\|_{p,\sigma}\leq \|\sigma\|_{\infty}^{1/p}\|u\|_{p}}$ valid for  every
$u\in L^p(\Omega)$ and the density of
$C^\infty(\Omega)\cap W^{1,p}_0(\Omega)$~in~$W^{1,p}_0(\Omega)$.

Another possible approach is to define the weighted Sobolev space $\Wpw$ as
the set of all functions $u\in \Lpw$ which posses a distributional gradient
${\nabla u\! \in\!\Lpw}$. We equip $\Wpw$ with the norm
$\Vert\cdot\Vert_{1,p,\sigma}$.~Note~that,~in~general, the space $\Wpw$ need not to
be a Banach space (cf.~\cite{heinonen}). To make $\Wpw$ a Banach space, the condition $\smash{\sigma
^{\frac {-1}{p-1}} \in L^1_\loc (\Omega)}$ is sufficient (cf.~\cite{kufner_opic}).~\mbox{However}, this
condition is for our purposes too restrictive (cf.~Section
\ref{sec:E}).~As~a~consequence, we will not use $\Wpw$, but we will work with the
spaces $\Hpw$.


\subsection{Auxiliary results}\label{sec:aux}

The following generalization of a classical result (cf.~\cite{ggz}) is
very useful~in~the identification of limits.
\begin{Sa}\label{pfastue}
  Let $\Omega\subseteq\R^d$,
  $d\in\mathbb{N}$, be a bounded domain,
  $\sigma$ a weight~and ${p\in [1,\infty)}$. Then, for a sequence
  $(u_n)_{n\in \mathbb{N}}\subseteq L^p(\Omega;\sigma)$
  from\footnote{Recall that $\nu_\sigma(A)=\int_A\sigma\, dx$ for all measurable sets $A\subseteq \Omega$.}
  \begin{enumerate}
  \item [{\rm (i)}]$\lim\limits_{n\to\infty} u_n=v$ $\nu_\sigma$--a.e.~in
    $\Omega$,
  \item [{\rm (ii)}] $u_n\rightharpoonup u$ in $L^p(\Omega;\sigma)$ $(n\to \infty)$,
  \end{enumerate}
  it follows that $u=v$ in $L^p(\Omega;\sigma)$.
\end{Sa}

\begin{Bew}
	See \cite[Thm.~13.44]{Hew-strom-65}.
\end{Bew}

Our proof relies on the following version of the Lipschitz truncation technique:
\begin{Sa}
  \label{thm:Lt}
  Let $G\hspace*{-0.1em}\subseteq\hspace*{-0.1em} \mathbb{R}^d$,
  $d\hspace*{-0.1em}\in\hspace*{-0.1em} \mathbb{N}$, be a bounded
  Lipschitz
  domain~and~${p\hspace*{-0.1em}\in\hspace*{-0.1em}(1,\infty)}$. Furthermore,
  let $\bfu^n \in W^{1,p}_0(G)$ be such that
  $\bfu^n \weakto \bfzero$ in $W^{1,p}_0(G)$ $(n\to \infty)$.  Then,
  for any ${j, n\in \setN}$, there
  exist
  $\bfu^{n,j}\in W^{1,\infty}_0(G)$
  and~$\smash{\lambda_{n,j}\in\big
    [2^{2^j}, 2^{2^{j+1}}\big ]} $~such~that
  \begin{align}
    \begin{split}
  \smash{ \lim_{n\to \infty}} \big ( {\sup}_{j \in \setN}
   \norm{\bfu^{n,j}}_{\infty}\big ) &=0\,,\\
    \norm{\nabla \bfu^{n,j}}_{\infty} &\leq c\, \lambda_{n,j}\leq c\,
    \smash{2^{2^{j+1}}}\,,
    \\
    \bignorm{\nabla \bfu^{n,j}\, \chi_{
        \set{\bfu^{n,j} \not= \bfu^n}}}_p^p &\leq c\, \lambda_{n,j}^p \, \vert\set{\bfu^{n,j} \not= \bfu^n}\vert \,,
    \\
    \smash{\limsup _{n \to \infty} }\,\lambda_{n,j}^p \,
      \vert\set{\bfu^{n,j} \not= \bfu^n}\vert &\leq c\, 2^{-j}\,,
   \end{split}\label{eq:C18}
  \end{align}
  where $c=c(d,p,G)>0$.  Moreover, for any $j \in
  \setN$, $\nabla \bfu^{n,j} \weakto \bfzero $~in~$L^s(G) $~${(n \to
  \infty)}$,
   $s \in [1,\infty)$, and $\smash{\nabla \bfu^{n,j} \stackrel{*}{\weakto}
  \bfzero}$ in $L^\infty(G)$ $(n \to
  \infty)$. 
\end{Sa}
\begin{Bew}
  See \cite[Theorem~2.5]{dms}.
\end{Bew}
Except  classical Korn's and \Poincare's inequalities, we also need the
following result for the divergence equation. 
\begin{Sa}\label{bog}
  Let $G\subseteq \mathbb{R}^d$, $d\ge 2$, be a bounded Lipschitz
  domain. Then, there exists a linear operator
  $\mathcal{B}_G:C^\infty_{0,0}G)\to C^{\infty}_{0}(G)$ which for all
  ${p\in(1,\infty)}$ extends uniquely to a linear, bounded operator
  $\mathcal{B}_G:L^p_0(G)\to W^{1,p}_0(G) $ such that
  $\|\mathcal{B}_Gu\|_{1,p}\leq c\,\|u\|_{p}$ and
  $\textup{div}\,\mathcal B_Gu = u$ for every $u\in L^p_0(G)$.
\end{Sa}
\begin{Bew}
	See \cite{bo1}, \cite{bo2}.
\end{Bew}

\section{The electric field $\bE$}\label{sec:E}

We first note that the system \eqref{maxwell}  is
separated from \eqref{NS}, in the sense that one can first solve the quasi-static
Maxwell's equations yielding an electric~field~$\bE$, which then, in turn, enters into \eqref{NS} as a parameter through the stress tensors.

It is proved in \cite{pi81}, \cite{pi84}, \cite{rubo}, that for bounded Lipschitz domains,~there
exists a solution\footnote{Here, we employ the standard function spaces
  $H(\curl):=\set{\bv \in L^2(\Omega)\fdg \curl \bv \in L^2(\Omega)}$,
  $H(\divo):=\set{\bv \in L^2(\Omega)\fdg \divo \bv \in L^2(\Omega)}$
  and $H^{-1/2}(\partial \Omega):= (H^{1/2}(\partial \Omega))^*$.}
$\E \in H(\curl) \cap H(\divo)$ of the system (\ref{maxwell})
with ${\norm{\bE}_2 \le c\, \norm{\bE_0}_{H^{-1/2}(\partial \Omega)}}$.
A more detailed analysis of the properties of the electric field $\bE$
can be found~in~\cite{frank-phd}. Let us summarize these results here.
First, note that combining $\eqref{maxwell}_1$ and $\eqref{maxwell}_2$, we obtain that
\begin{align*}
  -\Delta\mathbf{E}=\curl\curl \mathbf{E}-\nabla \Div\mathbf{E}=0\,,
\end{align*}
i.e.,~the electric field $\mathbf{E}$ is a harmonic. Moreover, the structure of the stress
tensors (cf.~Assumption~\ref{VssS}, Assumption~\ref{VssN}) yields 
that the natural functional setting of our problem involves weighted
 Sobolev spaces, where the weight~is~given~by~$\abs{\bE}^2$. Using the
theory of harmonic functions is it shown in \cite[Sec.~3.2]{frank-phd}
that $\abs{\bE}^2$ belongs to the Muckenhoupt class
$\mathcal A_\infty$ and that, in general,
$\abs{\bE}^\frac {-2}{p-1}$ does not belong to $L^1_\loc(\Omega)$. Since for our
investigations it is more important to work with a Banach space than
that the gradient is a distributional gradient, we, hence, work~with~the~space $H^{1,p}(\Omega;\abs{\bE}^2)$ and not with
the space $W^{1,p}(\Omega;\abs{\bE}^2)$.

On the other hand, because any harmonic function is real analytic, one can
characterize its zero set as follows:

\begin{Lem}\label{Untermannigfaltigkeit}
  Let $\Omega\subseteq\R^d$, $d\in \mathbb{N}$, be a bounded domain and
  $u:\Omega\to\R$~a~non-trivial analytic function. Then, 
 $u^{-1}(0)$ is a union of $C^1$--manifolds~$(M_i)_{i=1,\cdots,m}$, $m\in \mathbb{N}$,
  with $\dim M_i\leq d-1$ for every $i=1,\cdots,m$, and $\abs{u^{-1}(0)}=0$.
\end{Lem}

\begin{Bew} See \cite{frank-phd}, \cite[Lem.~3.1]{erw}.
\end{Bew}

Finally, we observe that using the regularity theory for Maxwell's equations 
(cf.~\cite{gsch}, \cite{rubo}), one can give conditions on the
boundary data $\bE_0$ ensuring that the electric field $\bE$ is globally
bounded, i.e., $\|\bE\|_\infty \le c(\bE_0)$.
Based on these observations, we will make the following assumption on
the electric field $\bE$:
\begin{Vss}\label{VssE}
  The electric field $\mathbf{E}$ satisfies $\mathbf{E}\in C^\infty(\Omega) \cap
  L^\infty(\Omega)$ and the closed set $\abs{\bE}^{-1}(0)$ is a null
  set, i.e., 
   ${\Omega_0:= \{ x\in \Omega \fdg \vert
  \mathbf{E}(x)\vert>0\}}$~has~full~measure.
\end{Vss}

In the sequel, we do not use that $\bE$ is the solution of the
quasi-static Maxwell's equations~\eqref{maxwell}, but we will only use
Assumption \ref{VssE}. The following embedding will play a substantial
role in our investigation.

\begin{Sa}\label{compactnew}
  Let $\Omega\subseteq \R^d$,
  ${d\in \mathbb{N}}$,~be~open,~$p\in\left[1,\infty\right)$ and let
  Assumption~\ref{VssE} be satisfied. Set $\smash{p^*:=\frac{dp}{d-p}}$
  if $p<d$ and $p^*:=\infty$ if $p\ge d$. Then, for any open~set
  $\Omega'\subset\subset \Omega$~with~${\partial\Omega'\in C^{0,1}}$
  and   any $\smash{\alpha\ge 1+\frac{2}{p}}$, it holds
  \begin{align*}
    H^{1,p}(\Omega;\vert \mathbf{E}\vert^2) \vnor
    L^{r}(\Omega';\vert \bfE\vert^{\alpha r})
  \end{align*}
  with $r \in [1,p^*]$ if $p\neq d$ and $r \in [1,p^*)$ if $p=d$. 
\end{Sa}
\begin{Bew} The proof of this result is inspired by \cite{Ava79}.
  First, let ${u\hspace*{-0.1em}\in\hspace*{-0.1em} \mathcal{V}_{p,\vert \bfE\vert^2}}$~be~arbitrary.
  Due to $\Omega'\subset\subset \Omega$ and
  ${\bfE\in C^\infty(\Omega)}$, it holds
  $\vert \bfE\vert^\alpha\in C^1(\overline{\Omega'})$ for any
  $\alpha>1$.~In~fact,  $\vert \bfE\vert^\alpha\in
  C^1(\overline{\Omega'}\setminus(\vert
  \bfE\vert^{-1}(0)))$~holds since we have
  \begin{align*}
    \nabla \vert \bfE\vert^\alpha= \alpha\vert
    \mathbf{E}\vert^{\alpha -2}\nabla\bfE^{\top}\bfE\quad\text{ in
    }\overline{\Omega'}\setminus\vert \bfE\vert^{-1}(0)\,, 
  \end{align*}
  which can be extended continuously to all of  $\,\overline{\Omega'}$
  for any $\alpha\!>\!1$.  Apparently,~we~have
  ${u\vert \mathbf{E}\vert^\alpha\in L^p(\Omega')}$ with
  \begin{align*}
    \|u\vert \mathbf{E}\vert^\alpha\|_{L^p(\Omega')}^p\leq
    \|\mathbf{E}\|_{L^\infty(\Omega')}^{\alpha p-2}\|
    u\|_{L^p(\Omega';\vert \mathbf{E}\vert^2)}^p\,,
  \end{align*}
  since $\alpha p\ge 2$. Moreover, we have
  ${u\vert \mathbf{E}\vert^\alpha\in W^{1,p}(\Omega')}$. In fact,
  due~to~${\alpha \ge 1+\frac{2}{p}}$,~and
  $ \nabla (u\vert \mathbf{E}\vert^\alpha)= \nabla u\vert
  \mathbf{E}\vert^\alpha+ u\alpha \vert
  \mathbf{E}\vert^{\alpha-2}\nabla\bfE^{\top}\bfE$ almost everywhere
  in~$\Omega'$, we get
  \begin{align*}
    &\|\nabla (u\vert
      \mathbf{E}\vert^\alpha)\|_{L^p(\Omega')}^p\leq2^p\big(\|\nabla
      u\vert\mathbf{E}\vert^\alpha\|_{L^p(\Omega')}^p+\alpha^p\|\nabla
      \bfE\|_{L^\infty(\Omega')}^p\| u\vert \mathbf{E}\vert^{\alpha
      -1}\|_{L^p(\Omega')}^p\big)
    \\
    &\leq 2^p\big(\|\mathbf{E}\|_{L^\infty(\Omega')}^{\alpha p -2}\| \nabla
    u\|_{L^p(\Omega';\vert \mathbf{E}\vert^2)}^p+\|\mathbf{E}\|_{L^\infty(\Omega')}^{(\alpha
    -1)p -2} \alpha^p\|\nabla \bfE\|_{L^\infty(\Omega')}^p\|
    u\|_{L^p(\Omega';\vert \mathbf{E}\vert^2)}^p\big)\,. 
  \end{align*}
  Hence, Sobolev's embedding theorem yields
  a~constant~${c_S>0}$~such~that~we~have for the above specified
  exponents $r$
  \begin{align}
   \|u\|^p_{L^{r}(\Omega';\vert\mathbf{E}\vert^{\alpha r})}&= \|u\vert \mathbf{E}\vert^\alpha
    \|_{L^{r}(\Omega')}^p\leq c_S\|u\vert
      \mathbf{E}\vert^\alpha\|_{W^{1,p}(\Omega')}^p\label{embedd}
    \\
    &\leq  c_S2^p\big(\|\mathbf{E}\|_{L^\infty(\Omega')}^{\alpha
    p-2}\!+\!\|\mathbf{E}\|_{L^\infty(\Omega')}^{(\alpha
    -1)p-2}\alpha^p\|\nabla
    \bfE\|_{L^\infty(\Omega')}^p\big)\|u\|_{H^{1,p}(\Omega',\vert
    \mathbf{E}\vert^2)}^p. \notag
  \end{align}
  Next, let $u\in H^{1,p}(\Omega,\vert \mathbf{E}\vert^2)$ be
  arbitrary. Then, by definition, there is~a~sequence
  $(u_n)_{n\in \mathbb{N}}\subseteq \mathcal{V}_{p,\vert \bfE\vert^2}$
  such that $u_n\to u$ in $H^{1,p}(\Omega,\vert \mathbf{E}\vert^2)$
  $(n\to \infty)$. Thus, resorting to inequality \eqref{embedd}, it is
  readily seen that
  $(u_n)_{n\in \mathbb{N}}\subseteq \mathcal{V}_{p,\vert \bfE\vert^2}$
  is a Cauchy sequence in
  $L^{r}(\Omega';\vert\mathbf{E}\vert^{\alpha r})$. Since
  $L^{r}(\Omega';\vert\mathbf{E}\vert^{\alpha r})$ is
  complete,~there~exists~some~${v\in
    L^{r}(\Omega';\vert\mathbf{E}\vert^{\alpha r})}$ such that
  $u_n\to v$ in $L^{r}(\Omega';\vert\mathbf{E}\vert^{\alpha r})$
  $(n\to \infty)$. To identify $u$ with $v$, one usually uses the
  embeddings
  $L^{r}(\Omega';\vert\mathbf{E}\vert^{\alpha r}), L^p(\Omega',\vert
  \mathbf{E}\vert^2)\hookrightarrow L^1_{\loc}(\Omega')$. However,~in~general, we do not have these embeddings available and need to argue
  differently.~We~exploit that from $u_n\hspace*{-0.1em}\to\hspace*{-0.1em} u$ in
  $H^{1,p}(\Omega,\vert \mathbf{E}\vert^2)$ $(n\hspace*{-0.1em}\to\hspace*{-0.1em} \infty)$ and
  $u_n\hspace*{-0.1em}\to\hspace*{-0.1em} v$~in~$L^{r}(\Omega';\vert\mathbf{E}\vert^{\alpha r})$~${(n\hspace*{-0.1em}\to\hspace*{-0.1em} \infty)}$, it follows that, up to a subsequence, it holds $u_n\to u$ 
  $\nu_{\smash{\abs{\bE}^2}}$--a.e.~in $\Omega'$ $(n\to \infty)$ and $u_n\to v$ 
  $\nu_{\smash{\abs{\bE}^{\alpha r}}}$--a.e.~in $\Omega'$ $(n\to \infty)$. 
  The
  properties of $\bE$ and Tschebyscheff’s inequality  imply that the
  Lebesgue measure is absolutely continuous with respect to the
  measures $\nu_{\smash{\abs{\bE}^2}}$ and $\nu_{\smash{\abs{\bE}^{\alpha r}}}$. Therefore, we
  conclude~that~${u=v}$~a.e.~in~$\Omega'$. Since
  $\nu_{\smash{\abs{\bE}^{\alpha r}}}$ is
  also absolutely continuous with respect to the Lebesgue measure, we just
  proved $u=v$ in $L^{r}(\Omega';\vert\mathbf{E}\vert^{\alpha r})$.
 \end{Bew}

\begin{Lem}\label{hatgrad}
	Let $\Omega\subseteq \R^d$,
	${d\in \mathbb{N}}$,~be~open,~$p\in\left[1,\infty\right)$ and let
	Assumption~\ref{VssE} be satisfied.  Then,
	for any $\Omega'\subset\subset\Omega_0$, we have that
	 ${W^{1,p}(\Omega')=H^{1,p}(\Omega';\vert\bfE\vert^2)}$ with norm
	equivalence (depending on $\Omega'$~and~$\bfE$) and
	$\hat{\nabla} u=\nabla u$ for all ${u\in W^{1,p}(\Omega')}$. 
	\end{Lem}

	\begin{proof}
          Due to $\vert \bE\vert >0 $ in $\overline{\Omega'}$ and
          ${\vert \bE\vert\in C^0(\overline{\Omega'})}$, there is a
          local constant $c(\Omega')>0$ such that
          $c(\Omega')^{-1}\leq \vert \bE\vert^2\leq
          c(\Omega')$~in~$\overline{\Omega'}$. Thus, we have
          ${L^p(\Omega')=L^p(\Omega';\vert \bE\vert^2)}$~with
          \begin{align*}
            \smash{c(\Omega')^{-\frac{1}{p}}\|u\|_{L^p(\Omega')}\leq
            \|u\|_{L^p(\Omega';\vert \bE\vert^2)}\leq
            c(\Omega')^{\frac{1}{p}}\|u\|_{L^p(\Omega')}} 
          \end{align*}
          for every $u\in L^p(\Omega')=L^p(\Omega';\vert
          \bE\vert^2)$. As a result, we also have
          $\mathcal{V}_{p,\vert \bE\vert^2}=\mathcal{V}_{p,1}$~with
          \begin{align}
            \smash{c(\Omega')^{-\frac{1}{p}}\|u\|_{W^{1,p}(\Omega')}\leq
            \|u\|_{H^{1,p}(\Omega';\vert \bE\vert^2)}\leq
            c(\Omega')^{\frac{1}{p}}\|u\|_{W^{1,p}(\Omega')}}\label{equivalence} 
          \end{align}
          for every
          $u \in \mathcal{V}_{p,\vert
            \bE\vert^2}=\mathcal{V}_{p,1}$. Since $W^{1,p}(\Omega')$,
          by Meyer--Serrin's theorem, is the closure of
          $\mathcal{V}_{p,1}$ and $H^{1,p}(\Omega';\vert\bfE\vert^2)$,
          by definition, is the closure of
          $\mathcal{V}_{p,\vert \bE\vert^2}$, 
          \eqref{equivalence} implies that
          ${W^{1,p}(\Omega')=H^{1,p}(\Omega';\vert\bfE\vert^2)}$ and
          $\smash{\hat{\nabla}} u=\nabla u$ for all
          ${u\in W^{1,p}(\Omega')}$.
	\end{proof}

\section{A weak stability lemma}\label{sec:stab}
The weak stability of problems of $p$--Laplace type is well-known
(cf.~\cite{dms}).~It~also holds for our problem \eqref{NS} if we make
appropriate natural assumptions on the extra stress tensor $\mathbf{S}$ and on the couple stress tensor
$\mathbf{N}$, which are motivated by the canonical example
in~\eqref{eq:SN-ex} for constant shear exponents.~We~denote the symmetric and the skew-symmetric part, resp., of a
tensor $\mathbf{A}\in \mathbb{R}^{d\times d}$ by $\mathbf{A}^{\sym}:=\frac{1}{2}(\mathbf{A}+\mathbf{A}^\top)$
and ${\mathbf{A}^{\anti}:=\frac{1}{2}(\mathbf{A}
-\mathbf{A}^\top)}$.~Moreover,~we~define~${\R^{d\times
	d}_{\sym}:=\{\mathbf{A}\in \R^{d\times d} \fdg
\mathbf{A}=\mathbf{A}^\sym\}}$ and $\R^{d\times
	d}_{\anti}:=\{\mathbf{A}\in \R^{d\times d} \fdg
      {\mathbf{A}=\mathbf{A}^\anti}\}$.
      
\begin{Vss}\label{VssS}
	For the extra stress tensor
	$\mathbf{S}:\R_{\sym}^{d\times d}\times \R_{\anti}^{d\times
          d}\times\R^d\to \setR^{d}$  and some $p \in(1,\infty)$, there
        exist constants $c,C >0$ such that:
	\begin{enumerate}
		\item[{\rm \hypertarget{(S.1)}{(S.1)}}] $\mathbf{S}\in C^0(\R_{\sym}^{d\times d}\times \R_{\anti}^{d\times d}\times \R^d;\setR^{d
			\times d})$.
		\item[{\rm \hypertarget{(S.2)}{(S.2)}}] For every $\bD \in \R_{\sym}^{d\times d}$,
		$\bR \in \R_{\anti}^{d\times d}$ and $\bE \in \setR^d$, it holds\\[-5mm]
		\begin{align*}
				\vert\mathbf{S}^{\sym}(\mathbf{D},\mathbf{R},\mathbf{E})\vert&\leq
				c\,\big (1+\vert\E\vert^2 \big) \big (1+\vert\mathbf{D}\vert^{p-1}\big)\,,
				\\ 
				\vert\mathbf{S}^{\anti}(\mathbf{D},\mathbf{R},\mathbf{E})\vert&\leq
				c\,\vert \mathbf{E}\vert^2 \big (1+\vert\mathbf{R}\vert^{p-1}\big)\,.
		\end{align*}\\[-7mm]
		
		\item[{\rm \hypertarget{(S.3)}{(S.3)}}]  For every $\bD \in \R_{\sym}^{d\times d}$,
		$\bR \in \R_{\anti}^{d\times d}$ and $\bE \in \setR^d$, it holds\\[-5mm]
		\begin{align*}
				\mathbf{S}(\mathbf{D},\mathbf{R},\mathbf{E}):\mathbf{D}
				&\geq  c\,\big (1+\vert\E\vert^2\big )\, \big( \vert\mathbf{D}\vert^p-C\big)\,,
				\\
				\mathbf{S}(\mathbf{D},\mathbf{R},\mathbf{E}):\mathbf{R}&\geq
				c\,\vert \mathbf{E}\vert^2 \big( \vert\mathbf{R}\vert^p-C\big)\,.
		\end{align*}\\[-7mm]
		
		\item[{\rm \hypertarget{(S.4)}{(S.4)}}]  For every $\bD_1, \bD_2 \in
		\R_{\sym}^{d\times d}$, $\bR_1, \bR_2 \in \R_{\anti}^{d\times d}$
		and $\bE \in \setR^d$ with $(\mathbf{D}_1,\vert
		\mathbf{E}\vert \mathbf{R}_1)\neq(\mathbf{D}_2,\vert
		\mathbf{E}\vert \mathbf{R}_2)$, it holds\\[-6mm]
		\begin{align*}
				\big (\mathbf{S}(&\mathbf{D}_1,\mathbf{R}_1,\mathbf{E})-
				\mathbf{S}(\mathbf{D}_2,\mathbf{R}_2,\mathbf{E})\big ):
				\big (\mathbf{D}_1-\mathbf{D}_2+\mathbf{R}_1-\mathbf{R}_2\big )>0\,.
		\end{align*}
	\end{enumerate}
\end{Vss}
\begin{Vss}\label{VssN}
	For the couple stress tensor
	$\mathbf{N}:\R^{d\times d}\times \R^d\to\R^{d\times d}$ and
        some $p \in (1,\infty)$, there
        exist constants $c,C >0$ such that: 
	\begin{enumerate}
		\item[{\rm \hypertarget{(N.1)}{(N.1)}}]$\mathbf{N}\in C^0(\R^{d\times d}\times \R^d;\R^{d\times d})$.
		
		\item[{\rm \hypertarget{(N.2)}{(N.2)}}]  For every $\bL \in \R^{d\times d}$ and $\bE \in
		\setR^d$, it holds\\[-6mm]
		\begin{align*}
			\vert\mathbf{N}(\mathbf{L},\mathbf{E})\vert\leq c\,\big
			\vert\mathbf{E}\vert^2
			\big(1+\vert\mathbf{L}\vert^{p-1}\big )\,.
		\end{align*}\\[-10mm]
		
		\item[{\rm \hypertarget{(N.3)}{(N.3)}}] For every $\bL \in \R^{d\times d}$ and $\bE \in
		\setR^d$, it holds\\[-6mm]
		\begin{align*}
			\mathbf{N}(\mathbf{L},\mathbf{E}):\mathbf{L}\geq
			c\,\big \vert\mathbf{E}\vert^2
			\big(\vert\mathbf{L}\vert^p -C\big)\,.
		\end{align*}\\[-10mm]

              \item[{\rm \hypertarget{(N.4)}{(N.4)}}]  For every $\bL_1,\bL_2 \in \R^{d\times d}$ and $\bE \in
		\setR^d$ with $\vert \mathbf{E}\vert>0$ and
		$\mathbf{L}_1\neq \mathbf{L}_2$, it holds\\[-6mm]
		\begin{align*}
			(\mathbf{N}(\mathbf{L}_1,\mathbf{E})-\mathbf{N}
			(\mathbf{L}_2,\mathbf{E})):(\mathbf{L}_1-\mathbf{L}_2)>0\,.
		\end{align*}
	\end{enumerate}
\end{Vss}

Under these assumptions, the following weak stability of 
our problem \eqref{NS} is valid. 
\begin{Lem}\label{DalMaso3}
  Let $\Omega\subseteq\R^d$,
  $d\ge2$, be a bounded domain,
  let $p> \frac {2d}{d+2}$ and
  let Assumption~\ref{VssS}, Assumption~\ref{VssN}~and Assumption
  ~\ref{VssE} be satisfied. Furthermore, let
  ${(\bv^n)_{n\in \mathbb{N}}\subseteq V_p(\Omega)}$ and
  $(\w^n)_{n\in \mathbb{N}}\subseteq \WpEo$ be~such~that
  \begin{align}
    \begin{aligned}\label{stab-konv}
      \bv^n&\rightharpoonup\bv\quad& &\text{in}\
      V_p(\Omega)&&\quad(n\to \infty)\,,
      \\
      \w^n&\rightharpoonup\w\quad& &\text{in}\ \WpEo&&\quad(n\to \infty) \,.
    \end{aligned}
  \end{align}      
  For a ball $B\subset \Omega_0$ such that
  $B':=2B\subset \subset \Omega_0$ and  $\tau\in C_0^\infty(B')$
  satisfying ${\chi_B\leq \tau\le\chi_{B'}}$ we set
  $ \mathbf{u}^n:=(\mathbf{v}^n-\mathbf{v})\tau\in W^{1,p}_0(B')$,
  ${\bfpsi^n:=(\boldsymbol\omega^n-\boldsymbol\omega)\tau \in
  W^{1,p}_0(B')}$, $n\in \mathbb{N}$.  Let
  $ \mathbf{u}^{n,j}\in W^{1,\infty}_0(B')$, $n,j\in \mathbb{N}$, and
  $\bfpsi^{n,j}\in W^{1,\infty}_0(B')$, $n,j\in \mathbb{N}$, resp.,
  denote the Lipschitz truncations constructed according to
  Theorem~\ref{thm:Lt}.~Moreover, assume that for every ${j\in \mathbb{N}}$, we
  have that
  \begin{align} 
    \limsup_{n\to\infty}\big\vert\big\langle
    &\Sn-\Ss,  \D\bu^{n,j}+\Rr(\bu^{n,j},\bfpsi^{n,j})\big\rangle \notag
    \\
    &\quad + \big\langle\Nn-\N,\nabla\bfpsi^{n,j}\big\rangle
      \big \vert \le
      \delta_j\,, \label{mon4}
  \end{align}
  where $\delta_j\to 0$ $(j\to \infty)$. Then, one has that
  $ \nabla\bv^n\to\nabla\bv$ a.e.~in $B$~${(n\to \infty)}$,
  ${\nabla\w^n\to\nabla\w}$ a.e.~in $B$ $(n\to \infty)$ and
  $\w^n\to\w$ a.e.~in $B$ $(n\to \infty)$~for~suitable subsequences.
\end{Lem}

\begin{Bem}\label{rem:hatgrad}
  For each ball $B'\subset\subset\Omega_0$, Lemma \ref{hatgrad} shows
  that
  $W^{1,p}(B')=H^{1,p}(B';\vert\bfE\vert^2)$. Hence, for any
  $\w \hspace*{-0.05em}\in\hspace*{-0.05em}
  H^{1,p}(\Omega;\vert\bfE\vert^2)$, it holds
  $\w|_{B'}\hspace*{-0.05em}\in\hspace*{-0.05em} W^{1,p}(B')$~with
  ${\nabla(\w|_{B'})=(\hat{\nabla}\w)|_{B'}}$ for each ball
  $B'\subset\subset\Omega_0$. 
	 In precisely that sense, the gradients of  $\w\in \WpEo$ and ${(\w^n)_{n\in \mathbb{N}}\subseteq \WpEo}$ are to be understood~in~\eqref{mon4}.
\end{Bem}

\begin{Bew}
	Since $W^{1,p}(B')=H^{1,p}(B',\vert\bfE\vert^2)$ with norm equivalence (cf.~Lemma~\ref{hatgrad}), from $\eqref{stab-konv}_2$ and resorting to Rellich's compactness theorem, we deduce that
	\begin{align}
		\begin{aligned}\label{konvergenz-stab}
			\bv^n&\to\bv\quad& &\text{in}\  L^q(B')\text{ and a.e.~in }
			B' &&\quad(n\to \infty)\,,
			\\
			\w^n&\to\w\quad& &\text{in}\ L^q(B') \text{ and a.e.~in }
			B'&&\quad(n\to \infty)\,,
		\end{aligned}
	\end{align}
	for any $q\in \left[1,p^*\right)$.
	Throughout the proof, we will employ the~particular~notation
	\begin{align}
		\begin{aligned}\label{eq:not}
			\widetilde\bS&:=\Ss\,, \quad &&\mathbf{S}^n\,:=\Sn\,,
			\\
			\widetilde\bN&:=\N\,, \quad &&\mathbf{N}^n:=\Nn\,.
		\end{aligned}
	\end{align}
Using (\hyperlink{(S.2)}{S.2}), (\hyperlink{(N.2)}{N.2}),
Assumption \ref{VssE} and  \eqref{stab-konv}, we see that there
exists a constant ${K:=K(\|\bE\|_\infty)>0}$ (not depending on $n\in \mathbb{N}$) such that 
\begin{equation}
  \label{stab-est}
  \begin{aligned}
    \Vert\bv^n\|_{1,p}+\|\bv \Vert_{1,p} +\Vert\w^n\|_{1,
      p,\abs{\bE}^2}+\|\w\Vert_{1, p,\abs{\bE}^2}&\le K\,,
    \\
    \Vert\mathbf{S}^n\Vert_{p'}  +  \Vert\widetilde{\mathbf{S}}\Vert_{p'}+  \Vert
    (\mathbf{S}^n)^\anti\Vert_{\smash{p',\vert
    		\bfE\vert^{\smash{\frac{-2}{p-1}}}}}+\Vert
    \widetilde{\mathbf{S}}^\anti\Vert_{\smash{p',\vert
      \bfE\vert^{\smash{\frac{-2}{p-1}}}}} &\le K\,,
    \\
\Vert\mathbf{N}^n\Vert_{\smash{p',\vert
		\bfE\vert^{\smash{\frac{-2}{p-1}}}}}+\Vert\widetilde {\mathbf{N}}\Vert_{\smash{p',\vert
      	\bfE\vert^{\smash{\frac{-2}{p-1}}}}}&\le
    K\,.
  \end{aligned}
\end{equation}
	Recall that $\tau\in C_0^\infty(B')$ with
	$\chi_B \le \tau\leq \chi_{B'}$. Hence, using (\hyperlink{(S.4)}{S.4}) and~(\hyperlink{(N.4)}{N.4}),~we~get
	\begin{align}
		&I^n\!:=\!\int_B\! \big [\!\big (\bS^n\!-\!\widetilde \bS\big )\!:\! \big
		(\D(\bv^n\!-\!\bv)\!+\!\Rr(\bv^n\!-\!\bv,\w^n\!-\!\w) \big)
		\!+\!\big(\bN^n\!-\!\widetilde \bN\big )\!:\!\nabla(\w^n\!-\! \w)\big]^\theta\, dx \notag
		\\
		&\le\!\int_{B'}\!\big [\big (\bS^n\!-\!\widetilde \bS\big )\!:\!
		\big (\D(\bv^n\!-\!\bv)\!+\!\Rr(\bv^n\!-\!\bv,\w^n\!-\!\w) \big)
		\!+\!\big(\bN^n\!-\!\widetilde \bN\big )\!:\!\nabla(\w^n\!-\! \w) \big]^\theta \!\tau ^\theta \,dx \notag
		\\
		&\le\int_{B'}\big [\big (\bS^n-\widetilde \bS\big ):
		\big (\D(\bv^n-\bv)+\Rr(\bv^n-\bv,\w^n-\w) \big)
		\tau \big]^\theta \, dx \label{eq:I}
		\\
		&\quad +\int_{B'}\big [\big(\bN^n-\widetilde \bN\big
		):\nabla(\w^n- \w) \tau \big]^\theta 
		dx =:\int_{B'} \alpha_n^\theta  \,dx +\int_{B'}
		\beta_n^\theta \, dx \,,\notag 
	\end{align}
	where we also used that
        \begin{gather}
          \smash{\frac{1}{2} }(a^\theta+ b^\theta ) \le (a+ b )^\theta \le
          a^\theta+ b^\theta\label{eq:1}
        \end{gather}
        valid for all $a,b\ge 0$~and~${\theta \in (0,1)}$. Then,
	splitting the integral of $\alpha_n^\theta$~over~$B'$~into an
	integral over $\{\mathbf{u}^n\neq\mathbf{u}^{n,j}\} $ and one over
	$\{\mathbf{u}^n=\mathbf{u}^{n,j}\} $, also using~H\"older's~inequality 
	with exponents $\frac{1}{\theta}$ and $\frac{1}{1-\theta}$, 
	we~find~that
	\begin{align}
		\int_{B'} \alpha_n^\theta \, dx  &\le \| \alpha_n
		\|_{L^1( B')}^\theta
		\vert\{\mathbf{u}^n\neq\mathbf{u}^{n,j}\}\vert^{1-\theta} + \|
		\alpha_n \chi_{\{\mathbf{u}^n=\mathbf{u}^{n,j}\} }
		\|_{L^1(B')}^\theta\vert B'\vert^{1-\theta}\notag
		\\
		&=:(I_{1}^n)^{\theta} \vert\{\mathbf{u}^n\neq\mathbf{u}^{n,j}\}\vert^{1-\theta}
           + \|
		\alpha_n \chi_{\{\mathbf{u}^n=\mathbf{u}^{n,j}\} }
		\|_{L^1(B')}^\theta\vert B'\vert^{1-\theta} 
           \,.\label{eq:i1}
	\end{align}
	For the first term, we will use  \eqref{eq:C18}$_4$ and, thus, have to
	show~that~$(I_1^n)_{n\in \mathbb{N}}\subseteq \mathbb{R}$~is bounded. To this end, we use that for vector
	fields $\bu$, $\bw$ and tensor~fields~$\bA$  there holds
	\begin{align}\label{eq:symm}
		\bA: \bD\bu +\bA: \bR(\bu,\bw) = \bA: \nabla \bu +
		\bA^\anti:(\bfepsilon\cdot \bw )\,.
	\end{align}
	Then,  combining \eqref{stab-konv},
	\eqref{stab-est}, \eqref{eq:symm} and  that $\tau \le 1$ in $\Omega$, we observe that
	\begin{align}
          \begin{aligned}
            I_{1}^n &\le \big (\Vert\bS^n\Vert_{p'} + \Vert\widetilde
            \bS\Vert_{p'} \big ) \Vert\nabla \bv^n\!-\!\nabla \bv
            \Vert_{p}
            \\
            &\quad + \big (\Vert(\bS^n)^\anti\Vert_{\smash{p',\vert
            		\bfE\vert^{\smash{\frac{-2}{p-1}}}}}+ \Vert \widetilde
            \bS^\anti\Vert_{\smash{p',\vert \bfE\vert^{\smash{\frac{-2}{p-1}}}}} \big )
            \Vert\w^n\!-\!\w\Vert_{ p,\abs{\bE}^2} \\& \le 2\,K^2\,.
          \end{aligned}\label{eq:i1est}
	\end{align}
	Similarly, we deduce that
	\begin{align}
		\int_{B'} \beta_n^\theta \, dx  &\le \| \beta_n
		\|_{L^1( B')}^\theta
		\vert\{\bfpsi^n\neq\bfpsi^{n,j}\}\vert^{1-\theta}\notag
		+ \|\beta_n \chi_{\{\bfpsi^n=\bfpsi^{n,j}\}  }
		\|_{L^1(B')}^\theta\vert B'\vert^{1-\theta}
		\\
		&=:(I_{2}^n)^{\theta} \vert\{\bfpsi^n\neq\bfpsi^{n,j}\}\vert^{1-\theta}
		+ \|\beta_n \chi_{\{\bfpsi^n=\bfpsi^{n,j}\}  }
           \|_{L^1(B')}^\theta\vert B'\vert^{1-\theta}
           \,,\label{eq:i3}
	\end{align}
	and that
	\begin{align}
          \begin{aligned}
            I_{2}^n &\le \big (\Vert\bN^n\Vert_{\smash{p',\vert \bfE\vert^{\smash{\frac{-2}{p-1}}}}} +
            \Vert\widetilde \bN\Vert_{\smash{p',\vert \bfE\vert^{\smash{\frac{-2}{p-1}}}}}
            \big )\Vert\nabla \w^n-\nabla \w \Vert_{p,\abs{\bE}^2}
          \le K^2\,.
          \end{aligned}\label{eq:i3est}
	\end{align}
	Using \eqref{eq:i1}, \eqref{eq:i1est}--\eqref{eq:i3est} and \eqref{eq:1}
	we, thus, conclude that
	\begin{align}
		\begin{aligned}
            &\int_{B'} \alpha_n^\theta \, dx +\int_{B'} \beta_n^\theta \, dx
          \\
            &\le 2^\theta K^{2\theta} \big
            (\vert\{\bfu^n\neq\bfu^{n,j}\}\vert^{1-\theta}
            +\vert\{\bfpsi^n\neq\bfpsi^{n,j}\}\vert^{1-\theta} \big )
             \\
            &\quad + 2\,\vert B'\vert^{1-\theta} \, 
            \bigg(\int_{B'}{\alpha_n \chi_{\{\bfu^n=\bfu^{n,j}\}}\,dx}
            +\int_{B'}{\beta_n\,\chi_{\{\bfpsi^n=\bfpsi^{n,j}\}}\,dx}
            \bigg)^{\smash{\theta}}\,.
          \end{aligned}\label{eq:est-ab}
	\end{align}
	Let us now treat the last two integrals, which we denote by
	$I_3^{n,j}$ and $I_4^{n,j}$.  We~have $ \nabla (\bfv^n -\bfv) \tau = \nabla\bu^{n,j}-
	(\bv^n-\bv) \otimes \nabla \tau $ on $\{\mathbf{u}^n=\mathbf{u}^{n,j}\}
	$, which, using~\eqref{eq:symm},~implies
	\begin{align}
	\begin{aligned}
          I_3^{n,j}&= \bigskp{ \bfS^n- \widetilde \bfS}{\big(\nabla
			\bfu^{n,j}  -(
			\bfv^{n}-\bfv)\otimes \nabla \tau \big)
			\chi_{\{\mathbf{u}^n=\mathbf{u}^{n,j}\}}}
		\\[-0.5mm]
		&\quad +\bigskp{ \big(\bfS^n- \widetilde \bfS\big )^\anti
		}{\bfepsilon \cdot \bfpsi^{n,j} \chi_{\{\bfpsi^n=\bfpsi^{n,j}\}
		}}
		\\[-0.5mm]
		&\quad +\bigskp{ \big(\bfS^n- \widetilde \bfS\big )^\anti
		}{\bfepsilon \cdot (\w^n -\w)\tau 
			\chi_{\{\mathbf{u}^n=\mathbf{u}^{n,j}\}\cap
				\{\bfpsi^n\neq\bfpsi^{n,j}\} }}
		\\[-0.5mm]
		&\quad -\bigskp{ \big(\bfS^n- \widetilde \bfS\big )^\anti
		}{\bfepsilon \cdot (\w^n-\w)\tau 
			\chi_{\{\mathbf{u}^n\neq\mathbf{u}^{n,j}\}\cap
				\{\bfpsi^n=\bfpsi^{n,j}\}}}\,.
				\end{aligned}\label{eq:I5}
	\end{align}
	From $ \nabla (\w^n
	-\w) \tau \hspace*{-0.1em}=\hspace*{-0.1em} \nabla\bfpsi^{n,j}\hspace*{-0.1em}-\hspace*{-0.1em} (\w^n-\w) \otimes \nabla \tau $ on $\{\bfpsi^n\hspace*{-0.1em}=\hspace*{-0.1em}\bfpsi^{n,j}\} $, it
	follows~that
	\begin{align}
		\begin{aligned}
		I_4^{n,j}= \bigskp{ \bfN^n- \widetilde \bfN}{\big(\nabla \bfpsi^{n,j}-(
				\w^{n}-\w)\otimes \nabla \tau 
			\big)	\chi_{\{\bfpsi^n=\bfpsi^{n,j}\} }}\,.
		\end{aligned}\label{eq:I6}
	\end{align}
	Using \eqref{eq:symm} and adding  appropriate terms,
	we deduce from \eqref{eq:I5} and~\eqref{eq:I6}~that
	\begin{align}
		I^{n,j}_{3} \!+\! I^{n,j}_{4} &\le \bigabs{\bigskp{ \bfS^n-  \widetilde \bfS}{\big (\bfD
				\bfu^{n,j} + \Rr(\bu^{n,j},\bpsi^{n,j}) \big) }+\bigskp{
				\bfN^n- \widetilde \bfN}{\nabla \bfpsi^{n,j} } } \notag 
		\\
		&\quad +\bigabs{\bigskp{ \bfS^n- \widetilde \bfS}{\!\nabla
				\bfu^{n,j}\chi_{\{\mathbf{u}^n\neq\mathbf{u}^{n,j}\}}}} +
		\bigabs{\bigskp{ \bfN^n- \widetilde \bfN}{\nabla
				\bfpsi^{n,j} \chi_{\{\bfpsi^n\neq \bfpsi^{n,j}\} }} } \notag 
		\\
		& \quad + \bigabs{\bigskp{ \big(\bfS^n\!-\! \widetilde \bfS\big
				)^\anti }{\bfepsilon \cdot \bfpsi^{n,j}
				\chi_{\{\bfpsi^n\neq\bfpsi^{n,j}\} }}} \!+\! \bigskp{\bigabs{
				\big(\bfS^n\!-\! \widetilde \bfS\big )^\anti }}{\bigabs{
				\w^n\!-\!\w}\tau } \notag 
		\\
		&\quad +\bigskp {\bigabs{ \bfS^n-\widetilde \bfS}}{\bigabs{(
				\bfv^{n}-\bfv)\otimes \nabla \tau }} +\bigskp {\bigabs{
				\bfN^n-\widetilde \bfN}}{\bigabs{( \w^{n}-\w)\otimes
				\nabla \tau }} \notag 
		\\[-1mm]
		&=: 
		{\sum}_{k=5}^{11}I_k^{n,j}\,. 
		\label{eq:i56}
	\end{align}
	The term $I_5^{n,j}$, i.e., the first line on the right-hand side in
        \eqref{eq:i56}, is handled by \eqref{mon4}. For the other
        terms we~obtain,  using H\"older's inequality~and~\eqref{stab-est}, that
	\begin{align}
		I_6^{n,j}&\le \big (\Vert\bS^n\Vert_{p'} + \Vert\widetilde
		\bS\Vert_{p'}
		\big ) \Vert \nabla
		\bfu^{n,j}\chi_{\{\mathbf{u}^n\neq\mathbf{u}^{n,j}\}} \Vert_{
			L^p(B')} \notag
		\\
		&\le K\,\Vert
		\nabla \bfu^{n,j}\chi_{\{\mathbf{u}^n\neq\mathbf{u}^{n,j}\}}
		\Vert_{ L^p(B')} \,,\label{eq:i8}\\[1mm]
	I_7^{n,j}&\le \big (\Vert\bN^n\Vert_{\smash{p',\vert \bfE\vert^{\smash{\frac{-2}{p-1}}}}} + \Vert\widetilde \bN\Vert_{\smash{p',\vert \bfE\vert^{\smash{\frac{-2}{p-1}}}}} \big ) \Vert \nabla
	\bfpsi^{n,j}\chi_{\{\bfpsi^n\neq\bfpsi^{n,j}\}} \Vert_{
		L^p(B';\abs{\bE}^2)} \notag
	\\[-1mm]
	&\le  K\,	\|\bfE\|_{\infty}^{\frac{2}{p}}\Vert \nabla
	\bfpsi^{n,j}\chi_{\{\bfpsi^n\neq\bfpsi^{n,j}\}} \Vert_{
   L^p(B')} \,,\label{eq:i9}
          \\[0.5mm]
	I_{8}^{n,j}&\le \big (\Vert(\bS^n)^\anti\Vert_{
		\smash{p',\vert \bfE\vert^{\smash{\frac{-2}{p-1}}}}} + \Vert\widetilde
                      \bS\,^\anti\Vert_{\smash{p',\vert \bfE\vert^{\smash{\frac{-2}{p-1}}}}} \big )
                      \norm{\bE}^{\frac
                      2p}_{\infty}\abs{\Omega}^{\frac 1p} \Vert 
	\bfpsi^{n,j}\Vert_{L^\infty(B')}\notag
	\\[-1mm]
	&\le K \, \norm{\bE}^{\frac 2p}_{\infty}\abs{\Omega}^{\frac 1p} 
   \Vert \bfpsi^{n,j} \Vert_{ L^\infty(B')} \,,\label{eq:i10}
          \\[0.5mm]
		I_{9}^{n,j}&\le\big (\Vert(\bS^n)^\anti\Vert_{\smash{p',\vert \bfE\vert^{\smash{\frac{-2}{p-1}}}}} + \Vert\widetilde
	\bS\,^\anti\Vert_{\smash{p',\vert \bfE\vert^{\smash{\frac{-2}{p-1}}}}} \big )
	\Vert\w^n-\w \Vert_{L^p(B';\abs{\bE}^2)} \notag 
	\\[-1mm]
	&\le  K\,\|\bfE\|_{\infty}^{\frac{2}{p}}\Vert\w^n-\w \Vert_{
   L^p(B')} \,,\label{eq:i11}
          \\[0.5mm]
		I_{10}^{n,j}&\le\big (\Vert\bS^n\Vert_{p'} + \Vert\widetilde
	\bS\Vert_{p'}
	\big )\norm{\nabla \tau}_\infty\,\Vert\bv^n- \bv
	\Vert_{L^p(B')}\notag
	\\
	&\le  K\,\norm{\nabla \tau}_\infty\Vert\bv^n-\bv
   \Vert_{L^p(B')} \,,\label{eq:i12}
          \\[0.5mm]
	I_{11}^{n,j}&\le\big (\Vert\bN^n\Vert_{\smash{p',\vert \bfE\vert^{\smash{\frac{-2}{p-1}}}}} + \Vert\widetilde \bN\Vert_{\smash{p',\vert \bfE\vert^{\smash{\frac{-2}{p-1}}}}} \big ) \norm{\nabla
		\tau}_\infty\,\Vert\w^n-  \w \Vert_{
		L^p(B';\abs{\bE}^2)}\notag 
	\\[-1mm]
	&\le  K\,\norm{\nabla
		\tau}_\infty\,\|\bfE\|_{\infty}^{\frac{2}{p}}\Vert\w^n- \w
	\Vert_{ L^p(B')} \,.\label{eq:i13}
	\end{align}
	With \eqref{eq:C18}, \!\eqref{stab-konv}--
	\eqref{konvergenz-stab}
	and $1\hspace*{-0.2em}\le\hspace*{-0.2em} \lambda_{n,j}^p$, we get from \eqref{eq:I},\!
	\eqref{eq:est-ab}--\eqref{eq:i13}~for~all~${j\hspace*{-0.2em} \in\hspace*{-0.2em} \setN}$
	\begin{equation*}
                 \smash{ \limsup_{n\to \infty} I^n\le c\,\delta_j^\theta + c\,
                  K^{2\theta} \,2^{-j(1-\theta)}+ c\,(1+
                  \norm{\bE}_{\infty}^{\frac{2}{p}}\,)^{\theta}
                  K^{{\theta}} 2^{\frac{-j}p\theta}\,.}
	\end{equation*}
	Since $\lim_{j\to \infty} \delta_j  =0$, we observe that $I^n\to 0$ $(n\to \infty)$, which,~owing~to~${\theta \in (0,1)}$, (\hyperlink{(S.4)}{S.4}) and (\hyperlink{(N.4)}{N.4}),  implies for a suitable
	subsequence that
	\begin{align*}
		\big (\bS^n-\widetilde \bS\big ): \big
		(\D(\bv^n-\bv)+\Rr(\bv^n-\bv,\w^n-\w) \big) &\to 0
		\qquad \textrm{ a.e.~in }B&&\quad(n\to \infty)\,,
		\\
		\big(\bN^n-\widetilde \bN\big )\!:\!\big
		(\nabla\w^n-\nabla \w\big ) &\to 0 \qquad \textrm{ a.e.~in
		}B&&\quad(n\to \infty)\,.
	\end{align*}
	In view of \eqref{konvergenz-stab}, we also know that $\w ^n \to \w$
	a.e.~in~$B$~and,~hence,~we~can~conclude the assertion of
	Lemma \ref{DalMaso3} as in the proof of \cite[Lem.~6]{DalMaso2}.
\end{Bew}

\begin{Kor}\label{cor:ae-conv}
  Let the assumptions of Lemma \ref{DalMaso3} be satisfied for all
  balls ${B\subset \subset \Omega_0}$ with
  ${B':=2B \subset \subset \Omega_0}$.
  Then, we have for suitable subsequences that
  $\nabla \bv^n \to \nabla \bv$ a.e.~in $\Omega$ $(n\to \infty)$,
  $ \hat\nabla \w^n \to \hat\nabla \w$ a.e.~in $\Omega$
  $(n\to \infty)$ and $\w^n \to \,\w $ a.e.~in $\Omega$
  $(n\to \infty)$.
\end{Kor}
\begin{proof}
	Using all rational tuples contained
	in~$\Omega_0$ as centers, we find a countable family $(B_k)_{k
		\in \setN}$ of balls covering $\Omega_0$ such that $B'_k:=2B_k
	\subset \subset \Omega_0$ for~every~${k \in \setN}$. Using the usual
	diagonalization procedure, we construct suitable subsequences
	such that $\w^n \to \w$ a.e.~in $\Omega_0$ $(n\to \infty)$, $\hat\nabla\w^n\to\hat\nabla\w$\footnote{Here, we used again that $(\hat\nabla\w)|_{\smash{B_k'}}=\nabla(\w|_{\smash{B_k'}})$ in $B_k'$ for all $k\in \mathbb{N}$ according to Remark~\ref{rem:hatgrad}.} a.e.~in $\Omega_0$ $(n\to \infty)$~and
	$\nabla\bv^n\to\nabla\bv$ a.e.~in $\Omega_0$ $(n\to \infty)$. Since $\vert \Omega\setminus \Omega_0\vert=0$, we proved~the~assertion.
\end{proof}
\section{Existence theorem for constant shear exponents}\label{veroeffentlichung3}
\hspace*{-0.2em}Now we are prepared to prove our first main result,
namely the existence of \mbox{solutions} to the problem \eqref{NS}, \eqref{maxwell} for
$p>\frac {2d}{d+2}$ without imposing the additional assumption that $\vert\bfE\vert^2$ belongs to the Muckenhoupt class $\mathcal A_p$.\newpage
\begin{Sa}\label{thm:main4}
  Let $\Omega\subseteq\R^d$, $d\geq 2$, be a bounded domain, 
  let Assumption \ref{VssS} and Assumption~\ref{VssN} be satisfied for
  some  $p> \frac {2d}{d+2}$ and let Assumption~\ref{VssE}~be~satisfied. Then,~for~every
  ${\ff\hspace*{-0.12em}\in\hspace*{-0.12em}(\Wpo)^*\!\!}$ and
  $\bell\hspace*{-0.12em}\in\hspace*{-0.12em} (\WpEo)^*\!$, there
  exist $\mathbf{v}\hspace*{-0.12em}\in\hspace*{-0.12em}\Vp$ and
  $ {\w\hspace*{-0.12em}\in\hspace*{-0.12em} \WpEo}$ such that for
  every
  $\boldsymbol\varphi \hspace*{-0.05em}\in\hspace*{-0.05em}
  C^1_0(\Omega)$~with~${\divo \bphi
    \hspace*{-0.05em}=\hspace*{-0.05em}0}$ and for every
  $\boldsymbol\psi \in C_0^1(\Omega)$ with
  ${\nabla \boldsymbol\psi\in L^{\smash{\frac{q}{q-2}}}(\Omega;\vert
    \bfE\vert^{\smash{-\frac{\alpha q}{q-2}}})}$ for some
  ${q\in \left[1,p^*\right)}$, it holds
	\begin{align}
		\begin{aligned}
		&\big\langle \Ss-\bv\otimes\bv,\bD\boldsymbol\varphi+\bR(\bphi,\boldsymbol\psi)\big\rangle
		\\
		& \;+\big\langle \bN(\hat \nabla \w ,\bE)-\w\otimes\bv,\nabla\boldsymbol\psi\big\rangle
		=\big\langle
		\ff,\boldsymbol\varphi\big\rangle+\big\langle\bell,
		\boldsymbol\psi\big\rangle \,.
	\end{aligned}\label{NS2aa-ii-final}
	\end{align}
	Moreover, there exists a constant $c>0$ such that 
	\begin{equation*}
		 \smash{ 	\Vert\bv\Vert_{1,p} +\Vert\w\Vert_{1,p,\vert\E\vert^2} \leq
			c\, \big (1+\norm{\bE}_2+\Vert
			\ff\Vert_{\smash{(\Wpo)^*}}+\Vert\bell\Vert_{\smash{(\WpEo)^*}}\big
			)\,.}
	\end{equation*}
\end{Sa}

\begin{Bem}
	\begin{enumerate}
	\item[{\rm (i)}] The lower bound $\smash{p>\frac{2d}{d+1}}$ in
          \cite[Thm.~6.44]{erw} is improved~in Theorem \ref{thm:main4} to $\smash{p>\frac{2d}{d+2}}$.
	\item[{\rm (ii)}] In contrast to \cite[\!Thm.~\!5.49]{erw},
          \cite[\!Thm.~\!5.56]{erw} and \cite[\!Thm.~\!5.59]{erw},~we~do~not require in Theorem \ref{thm:main4} that
          $\vert\E\vert^2\in \mathcal A_p$. Even though we know that~~the~real analytic function $\vert\E\vert^2$ belongs to
          the Muckenhoupt~class~$\mathcal A_\infty$~(cf.~\mbox{Section~\ref{sec:E}}), this does not imply that $\vert\E\vert^2$ belongs
          to the Muckenhoupt~class~$\mathcal A_p$.
	
	\item[{\rm (iii)}] The results \cite[Thm.~5.56]{erw},
          \cite[Thm.~5.59]{erw} and \cite[Thm.~6.44]{erw}
          consider test functions
          $\boldsymbol\psi \in \vert \bfE\vert^\beta\times
          C_0^1(\Omega)$ for some $\beta\ge 2$~instead~of~${\boldsymbol\psi \in C_0^1(\Omega)}$ with
          $\nabla \boldsymbol\psi\hspace*{-0.1em}\in\hspace*{-0.1em}
          L^{\smash{\frac{q}{q-2}}}(\Omega;\vert
          \bfE\vert^{\smash{-\frac{\alpha q}{q-2}}})$ for some
          ${q\hspace*{-0.1em}\in\hspace*{-0.1em} \left[1,p^*\right)}$. However,~any~${\boldsymbol\psi\hspace*{-0.15em} \in\hspace*{-0.15em} C_0^1(\Omega_0)^d}$, satisfies both
          $\boldsymbol\psi \hspace*{-0.11em}\in\hspace*{-0.11em} \vert \bfE\vert^2\hspace*{-0.11em}\times\hspace*{-0.11em}
          C_0^1(\Omega)$ and
          ${\nabla \boldsymbol\psi\hspace*{-0.11em}\in\hspace*{-0.11em}
          L^{\smash{\frac{q}{q-2}}}(\Omega;\hspace*{-0.11em}\vert
          \bfE\vert^{\smash{-\frac{\alpha q}{q-2}}})}$~for~all~${q\hspace*{-0.11em}\in\hspace*{-0.11em} \left[1,p^*\right)}$.  Hence, if we assume that there exist 
          sufficiently smooth solutions
          ${\bfv:{\Omega}\to
          \mathbb{R}^d}$ and  $\w:{\Omega}\to
          \mathbb{R}^d$ of \eqref{NS2aa-ii-final} or \cite[(6.46)]{erw},
          then testing with $\boldsymbol\psi \in C_0^1(\Omega_0)$,
          integration by parts, the fundamental theorem of calculus of
          variations~and using
          that~${\vert
            \Omega\setminus\Omega_0\vert=0}$,~we~readily deduce~that
\begin{align*}
		\begin{aligned}
			-\Div \mathbf{S}+\Div(\mathbf{v}\otimes\mathbf{v})
			+\nabla\pi&=\mathbf{f} &&\text{in} \ \Omega\,,\\
			-\Div\mathbf{N}+\Div(\boldsymbol{\omega}\otimes
			\mathbf{v})&=\boldsymbol{\ell}-\bfvarepsilon:\mathbf{S}
			&\quad&\text{in} \ \Omega\,,
		\end{aligned}
\end{align*}
	i.e., the weak formulations in \cite[Thm. 5.56, Thm. 5.59 \& Thm. 6.44]{erw} and Theorem~\ref{thm:main4} yield comparable results.
	\end{enumerate}
\end{Bem}

\begin{proof}
  \textbf{1. Non-degenerate approximation and a-priori estimates:}\\
  Analogously to \cite[Thm.~6.1]{erw} or using the standard theory of
  pseudomonotone operators, one can show that for every
  $n\in \mathbb{N}$, there exist
  functions~$(\mathbf{v}^n,\boldsymbol\omega^n) \in {(\Vp\cap\Lr)}
  \times (\Wpo\cap\Lr)$ satisfying for every
  $\boldsymbol\varphi\in \Vp\cap\Lr$ and
  $\boldsymbol\psi\in \Wpo\cap\Lr$
	\begin{align}
		\hspace*{-0.1cm}\begin{aligned}
		&\big\langle \Sn	-\bv^n\otimes\bv^n,\bD\boldsymbol\varphi+\bR(\bphi,\boldsymbol\psi)\big\rangle+\tfrac 1n
		\big\langle\vert\mathbf{v}^n\vert^{r-2}\mathbf{v}^n,\bphi\big\rangle 
		\\
		& \;+\big\langle\bN^n_{\textup{non-deg}}-\w^n\otimes\bv^n,\nabla\boldsymbol\psi\big\rangle  +\tfrac 1n
		\big\langle\vert{\w}^n\vert^{r-2}{\w}^n,\bpsi\big\rangle=\big\langle
		\ff,\boldsymbol\varphi\big\rangle+\big\langle\bell,\boldsymbol\psi\big\rangle\,,
	\end{aligned} \label{NS2}
	\end{align}
	where $\bN^n_{\textup{non-deg}}:= \bN(\nabla \bomega^n,\bE) +\frac{1}{n}(1+\vert\nabla
	\boldsymbol{\omega}^n\vert)^{{p-2}{}}\nabla\boldsymbol{\omega}^n$
	and\footnote{We have chosen the exponent $r>2 p'$ such that both convective terms
	  $\skp{\bv\otimes \bv}{\nabla \bphi}$~and $\skp{\w\otimes
	    \bv}{\nabla \bphi}$ define compact operators from $\Lr\times \Lr$
	  to $(\Wpo)^*$.}
	$r>2p'$ is fixed. The existence of these solutions is for every $n\in \mathbb{N}$ based on the
        a-priori estimate 
	\begin{align} \label{apri0}
          \begin{aligned}
            &\int_{\Omega}{\big (1+ \abs{\bE}^2\big )\vert
              \bfD\bfv^n\vert^p+ \vert \nabla\w^n\vert^p\vert
              \bfE\vert^2+\vert \bR(\bfv^n,\w^n)\vert^p\abs{\bE}^2\,dx}
            \\
            &\quad +\smash{\int_{\Omega}{\tfrac{1}{n}\big(\vert
              \bfv^n\vert^r+\vert
              \nabla\w^n\vert^p+\vert\w^n\vert^r\big)\,dx} \le
            c\,, }
          \end{aligned}
	\end{align}
        which follows, using (\hyperlink{(S.3)}{S.3}) and (\hyperlink{(N.3)}{N.3}), in a
        standard way.  
	Using~Korn's~inequality~in the non-weighted case, the
        definition of $\bR(\bv, \w)$ and  Poincar\'e's inequality in the non-weighted case, we
        deduce, as in \cite[Sec.~4]{erw}, from \eqref{apri0} that~there~exists~a constant
        ${K \!:=\!K(\|\bE\|_2,\|\ff\|_{(\Wpo)^*},\|\bell\|_{(\WpEo)^*})\!>\!0}$~such~that~for~every~${n\in \mathbb{N}}$ 
	\begin{align}
		&
		\Vert\bv^n\Vert_{1,p}+\Vert\w^n\Vert_{1,p,\vert\E\vert^2}+\tfrac{1}{n}\Vert \mathbf{v}^n\Vert_r +
		\tfrac{1}{n} \Vert\boldsymbol\omega^n\Vert_{1,p}
		+\tfrac{1}{n}\Vert\boldsymbol\omega^n\Vert_r \leq K\,. \label{apri}
	\end{align}
Apart from that, using (\hyperlink{(S.2)}{S.2}),
(\hyperlink{(N.2)}{N.2}),~\eqref{apri}, $\bE \in L^\infty(\Omega)$ (cf.~Assumption \ref{VssE}) and the notation introduced in \eqref{eq:not}, we obtain for every $n\in \mathbb{N}$ that
	\begin{equation}
		\label{eq:est-SNn}
		\begin{aligned}\Vert\mathbf{S}^n\Vert_{p'}+
			\Vert (\mathbf{S}^n)^\anti\Vert_{\smash{p',\vert \bfE\vert^{\smash{\frac{-2}{p-1}}}}}+
			\Vert\mathbf{N}^n\Vert_{\smash{p',\vert \bfE\vert^{\smash{\frac{-2}{p-1}}}}}\le K\,.
		\end{aligned}
	\end{equation}

	\textbf{2. Extraction of (weakly) convergent subsequences:}\\
	The estimates \eqref{apri}, \eqref{eq:est-SNn} and Rellich's compactness theorem yield not relabeled subsequences as well as functions 
	${\bv\hspace*{-0.1em}\in \hspace*{-0.1em}V_p(\Omega)}$, ${\w\hspace*{-0.1em}\in\hspace*{-0.1em} \WpEo
}$, $\smash{\widehat \bS \hspace*{-0.1em}\in\hspace*{-0.1em} L^{p'}(\Omega)}$ and
	$\smash{\widehat \bN \hspace*{-0.1em}\in\hspace*{-0.1em}L^{p'}(\Omega;\vert\bfE\vert^{\frac{-2}{p-1}})}$ such that
	\begin{align}
			\bv^n&\rightharpoonup\bv\quad& &\text{in}\
			V_p(\Omega)&&\quad(n\to \infty)\,,\hphantom{\;}\notag
			\\
			\bv^n&\to\bv\quad& &\text{in}\ \Lq\text{ and a.e.~in
			}\Omega&&\quad(n\to \infty)\,,\label{konvergenz}
			\\[-1mm]
			\w^n&\rightharpoonup\w\quad& &\text{in}\ \WpEo &&\quad(n\to \infty)\,,\notag\\[1mm]
				\mathbf{S}^n&\rightharpoonup \widehat \bS \;& &\text{in}\
			L^{p'}(\Omega)&&\quad(n\to \infty)\,,\notag
			\\[-1mm]
			\hspace*{-1.75em}(\mathbf{S}^n)^{\anti}&\rightharpoonup \widehat \bS^\anti & &\text{in}\ L^{p'}(\Omega;\vert\bfE\vert^{\frac{-2}{p-1}})&&\quad(n\to \infty)\,,\label{konvergenz18}
			\\[-1mm]
			\mathbf{N}^n&\rightharpoonup \widehat\bN \quad& &\text{in}\
			L^{p'}(\Omega;\vert\bfE\vert^{\frac{-2}{p-1}})&&\quad(n\to \infty)\,,\notag
	\end{align}          
	where $q\in \left[1,p^*\right)$.

	\textbf{3. Identification of $\widehat\bS$ with $\Ss$~and~$\widehat\bN$~with~${\bfN(\hat\nabla\w,\bE)}$:}\\
	Recall that $\Omega_0= \{ x\in \Omega \fdg \vert
  \mathbf{E}(x)\vert>0\}$.	Next, let $B\subset\subset \Omega_0$ be a
  ball such that $B':=2B\subset \subset\Omega_0$. Then, due to Lemma~\ref{hatgrad}, we have  $W^{1,p}(B')=H^{1,p}(B',\vert\bfE\vert^2)$ with norm equivalence (depending on $B'$~and~$\bfE$). \!Therefore, from~$\eqref{konvergenz}_3$~and~Rellich's compactness theorem, we deduce that
	\begin{align}
		\begin{aligned}\label{konvergenz-w}
			\w^n&\rightharpoonup\w\quad& &\text{in}\  W^{1,p}(B')&&\quad(n\to \infty)\,,
			\\
			\w^n&\to\w\quad& &\text{in}\ L^q(B') \text{ and a.e.~in }
			B'&&\quad(n\to \infty)\,,
		\end{aligned}
	\end{align}
	where $q\in \left[1,p^*\right)$. In particular, this implies that $\w \in
	W^{1,p}(B')\cap L^q(B')$.~Next,~let $\tau\hspace*{-0.15em}\in\hspace*{-0.15em} C_0^\infty(B')$ satisfy $\chi_B \hspace*{-0.15em}\le\hspace*{-0.15em} \tau\hspace*{-0.15em}\leq\hspace*{-0.15em}
	\hspace*{-0.15em}\chi_{B'}$. 
        According to
	(\ref{konvergenz})$_{1}$ and (\ref{konvergenz-w})$_{1}$,~it~follows~that
	\begin{align}
		\begin{aligned}\label{konvergenz-w1}
			\mathbf{u}^n&:=(\mathbf{v}^n-\mathbf{v})\tau\rightharpoonup\mathbf{0}
			& &\text{in}\ W^{1,p}_0(B')&&\quad(n\to \infty)\,,
			\\
			\bfpsi^n&:=(\boldsymbol\omega^n-\boldsymbol\omega)\tau
			\rightharpoonup\mathbf{0} & &\text{in}\
			W^{1,p}_0(B')&&\quad(n\to \infty)\,.
		\end{aligned}
	\end{align}
	Denote for $n\in \mathbb{N}$, the Lipschitz truncation of $\bu^n\in W^{1,p}_0(B')$~and~${\bfpsi^n\in W^{1,p}_0(B')}$
	according to Theorem \ref{thm:Lt} with respect to the ball $B'$ by $(\mathbf{u}^{n,j})_{j\in \mathbb{N}}\!\subseteq\!
	W^{1,\infty}_0(B')$~and ${(\bfpsi^{n,j})_{j\in \mathbb{N}}\subseteq W^{1,\infty}_0(B')}$,
	resp. In particular, on the basis of \eqref{konvergenz-w1}, Theorem \ref{thm:Lt} implies that the Lipschitz truncations satisfy for
	every $j \in \setN$~and~${s \in [1,\infty)}$ 
	\begin{align}
		\begin{aligned}\label{Testfkt_2}
			\mathbf{u}^{n,j}&\rightharpoonup \mathbf{0}& &\text{in}\
			W^{1,s}_0(B')&&\quad(n\to \infty)\,,
			\\
			\mathbf{u}^{n,j}&\to \mathbf{0}& &\text{in}\
			L^{s}(B')&&\quad(n\to \infty)\,,
			\\
			\bfpsi^{n,j}&\rightharpoonup \mathbf{0} & & \text{in}\
                        W^{1,s}_0(B')&&\quad(n\to \infty)\,,
			\\
			\bfpsi^{n,j}&\to \mathbf{0}& &\text{in}\ 
			L^{s}(B')&&\quad(n\to \infty)\,,
                        \\
                        \bfpsi^{n,j}&\rightharpoonup \mathbf{0} & & \text{in}\
                        H^{1,s}_0(B';\abs{\bE}^2)&&\quad(n\to \infty)\,,
		\end{aligned}
	\end{align}
        where we used in the last line that $\bE\in L^\infty(\Omega)$ holds.
	Note that $\bfpsi^{n,j}\hspace*{-0.12em}\in\hspace*{-0.12em}
        W^{1,\infty}_0(B')$, $n,j\hspace*{-0.12em}\in\hspace*{-0.12em} \mathbb{N}$, are
        suitable test-functions~in~\eqref{NS2}.~\mbox{However}, 
        $\mathbf{u}^{n,j}\hspace*{-0.1em}\in\hspace*{-0.1em}
        W^{1,\infty}_0(B')$,
        $n,j\hspace*{-0.1em}\in\hspace*{-0.1em} \mathbb{N}$, are not
        admissible in \eqref{NS2} because they are not
        di\-vergence-free.~To~correct~this,  we
        define 
        $ \bw^{n,j}\hspace*{-0.2em}:=\hspace*{-0.1em}\mathcal{B}_{B'}(\divo
          \mathbf{u}^{n,j})$,~${n,j\hspace*{-0.1em}\in\hspace*{-0.1em}
          \mathbb{N}}$, where~${\mathcal{B}_{B'}\hspace*{-0.1em}:\hspace*{-0.1em}L^s_0(B')\hspace*{-0.1em}\to\hspace*{-0.1em} W^{1,s}_0(B')}$
        denotes the Bogovskii operator with respect~to~$B'\!$, ensured by
        Theorem~\ref{bog}.~Since~$\mathcal{B}_{B'}\!$~is~weakly continuous, \eqref{Testfkt_2}$_{1}$ and
        Rellich's compactness theorem~imply for
        every~${j\hspace*{-0.1em}\in\hspace*{-0.1em} \setN}$ and~${s\hspace*{-0.1em} \in\hspace*{-0.1em} (1, \infty)}$ that
	\begin{align}
		\begin{aligned}\label{Testfkt_3}
			\bw^{n,j}&\rightharpoonup \mathbf{0} & &\text{in}\ W^{1,s}_0(B')&&\quad(n\to \infty)\,,
			\\
			\bw^{n,j}&\to \mathbf{0} & &\text{in}\
			L^s(B')&&\quad(n\to \infty)\,.
		\end{aligned}
	\end{align}
Moreover, owing to the boundedness of $\mathcal{B}_{B'}$, one has for
any~${n,j \hspace*{-0.1em}\in \hspace*{-0.1em}\setN}$~and~${s
  \hspace*{-0.1em} \in \hspace*{-0.1em}  (1,\infty)}$ that
	\begin{equation}
		\label{eq:divo}
		\begin{aligned}
			\Vert\bw^{n,j}\Vert_{W^{1,s}_0(B')}&\leq
			c\,\Vert\Div\mathbf{u}^{n,j}\Vert_{L^s(B')}\,.
		\end{aligned}
	\end{equation}
	On the  basis of $\nabla \bfu^n = \nabla \bfu^{n,j}$ on the set $\set{\bfu^n =
		\bfu^{n,j} }$ (cf.~\cite[Cor.~1.43]{maly-ziemer})~and $\divo
	\bfu^n =\nabla \tau \cdot (\bv^n -\bv)$ for every~${n,j\in \setN}$, we further get for every~${n,j\in \setN}$ that
	\begin{equation}	\label{eq:divo2}
		\divergence \bfu^{n,j} = \chi_ { \set{\bfu^n \not= \bfu^{n,j} }}
		\divergence \bfu^{n,j} + \chi_ { \set{\bfu^n = \bfu^{n,j} }}\nabla
		\tau \cdot (\bv^n -\bv)\quad\text{ a.e.~in }B'\,.
	\end{equation}
	Then, \eqref{eq:divo} with $s=p$ and 	\eqref{eq:divo2}
        together imply for every~${n,j\in \setN}$ that
	\begin{align*}
		\norm{\bw^{n,j}}_{W^{1,p}_0(B')}\leq c\,
		\norm{\nabla \bfu^{n,j}\, \chi_ { \set{\bfu^n \not= \bfu^{n,j} } }
		}_{L^{p}(B')}  + c\,(\norm{\nabla \tau}_\infty ) \norm{\bv^n -\bv}_{L^{p}(B')}
		\,,
	\end{align*}
	which in conjunction with \eqref{eq:C18} and \eqref{konvergenz}$_1$  yields for every $j \in \setN$ that
	\begin{align}
		\begin{aligned}
			\smash{\limsup_{n\to \infty} \|\bw^{n,j}\|_{W^{1,p}_0(B')} \leq c\,2^{\frac {-j}p}}\,.
		\end{aligned}
		\label{eq:n3.29*a}
	\end{align}
	Setting
        $\boldsymbol\varphi^{n,j}:=\mathbf{u}^{n,j}-\bw^{n,j}$, ${n,j\in \setN}$,
	we observe that ${(\boldsymbol\varphi^{n,j})_{n,j\in
            \setN}\subseteq V_s(B')}$, $s \in (1,\infty)$, i.e., they are suitable test-functions in
	\eqref{NS2}. To use Corollary~\ref{cor:ae-conv}, we  have to verify that  condition
	\eqref{mon4} is satisfied. 
	 To this end, we test equation (\ref{NS2})
	with the admissible test-functions $\bphi=\boldsymbol\varphi^{n,j}$ and 
	$\bpsi=\bfpsi^{n,j}$ for every ${n,j\in \setN}$ and subtract on both sides 
	\begin{align*} 
		\big\langle\Ss, \D\bu^{n,j}+\Rr(\bu^{n,j},\bfpsi^{n,j})\big\rangle+
		\big\langle\bN(\hat\nabla\w,\bE),\nabla\bfpsi^{n,j}\big\rangle\,, \;\: n,j\in \setN\,. 
	\end{align*} 
	Owing to $\bphi^{n,j}=\bu^{n,j}-\bw^{n,j}$, 
        this yields~for~every~${n,j\in \setN}$~that 
	\begin{align}
		&\big\langle\bfS^n-\Ss,\,
		\D\bu^{n,j}+\Rr(\bu^{n,j},\bfpsi^{n,j})\big\rangle \notag
		+ \big\langle\bfN^n-\bN(\hat\nabla\w,\bE),\nabla\bfpsi^{n,j}\big\rangle \notag 
		\\
		&= \big\langle \ff,\boldsymbol\varphi^{n,j}\big\rangle +
		\big\langle\bell,\bfpsi^{n,j}\big\rangle -\tfrac 1n
		\skp{\vert\mathbf{v}^n\vert^{r-2}\mathbf{v}^n}{\bphi^{n,j}}
		+\big\langle
		\bv^n\otimes\bv^n,\nabla\boldsymbol\varphi^{n,j}\big\rangle  \notag
		\\
		&\quad -\tfrac 1n \bigskp{(1+ \abs{\nabla \w^n})^{p-2}\nabla
			\w^n}{\nabla\bfpsi^{n,j}} -\tfrac 1n
		\skp{\vert{\w}^n\vert^{r-2}{\w}^n}{\bfpsi^{n,j}} \notag
		\\
		&\quad
		+\big\langle\w^n\otimes\bv^n,\nabla\bfpsi^{n,j}\big\rangle
		+\big\langle \Sn,\nabla\bw^{n,j}\big\rangle \notag
		\\
		&\quad -\big\langle\Ss, \D\bu^{n,j}+\Rr(\bu^{n,j},\bfpsi^{n,j})\big\rangle
		-\big\langle\bN(\hat\nabla\w,\bE),\nabla\bfpsi^{n,j}\big\rangle \notag
		\\[-1mm]
		&=:
		{\sum}_{k=1}^{10} J_k^{n,j}\,.\label{eq:diff1}
	\end{align}

        \noindent On the basis of $\bv\in V_p(\Omega)$ and
        $\w\in \WpEo $, we
        get~using~(\hyperlink{(S.2)}{S.2})~and~(\hyperlink{(N.2)}{N.2})~that
        $\Ss\hspace*{-0.12em} \in\hspace*{-0.12em} L^{p'}(\Omega)$ and
        $\smash{\bN(\hat\nabla\w,\bE) \hspace*{-0.12em}\in\hspace*{-0.12em} L^{p'}(\Omega;\vert
          \bfE\vert^{\smash{\frac{-2}{p-1}}})}$~(cf.~\eqref{eq:est-SNn}).~Using~this, (\ref{Testfkt_2}) and (\ref{Testfkt_3}), we conclude for
        every $j \in \setN$ that
	\begin{align}\label{eq:j1} 
		\lim_{n\to\infty} J_1^{n,j} +J_2^{n,j} +J_9^{n,j} +J_{10}^{n,j} =0\,.
	\end{align}
	From  (\ref{apri}),
	(\ref{Testfkt_2}) and (\ref{Testfkt_3}), we obtain for every $j \in
	\setN$ that
	\begin{align}\label{eq:j3} 
		\lim_{n\to\infty} J_3^{n,j} +J_5^{n,j} +J_6^{n,j} =0\,.
	\end{align}
	Using the notation \eqref{eq:not}, the estimates \eqref{eq:est-SNn}
	and
	\eqref{eq:n3.29*a}, we get for every ${j\in \setN}$ that 
	\begin{align}
		\limsup_{n\to\infty} J_8^{n,j} \leq \limsup_{n\to\infty}\Vert \bfS^n\Vert_{p'} \Vert
		\nabla\bw^{n,j}\Vert_{L^p(B')}
		\leq
		c\,K\, 2^{\frac
			{-j}{p}}=:\delta_j\,.\label{eq:j8}
	\end{align} 
	From \eqref{konvergenz}$_2$ and
	\eqref{konvergenz-w}$_2$, it further  follows that %
	\begin{equation}
		\begin{split}\label{conv-convec} 
			\begin{aligned}
				\mathbf{v}^n\otimes \mathbf{v}^n&\to \mathbf{v}\otimes
				\mathbf{v}&& \text{in}\ L^{s'}(\Omega)&&\quad(n\to \infty)\,,
				\\
				\w^n\otimes \mathbf{v}^n&\to \w\otimes \mathbf{v}&&\text{in}\ L^{s'}(B')&&\quad(n\to \infty)\,,
			\end{aligned}
			\hspace*{5mm}
			\begin{aligned}
				s'\in \Big[1, \frac{p^*}{2}\Big)\,.
			\end{aligned}
		\end{split}
	\end{equation}
	Thus, combining (\ref{Testfkt_2}), (\ref{Testfkt_3}) and
        \eqref{conv-convec}, we find that for every $j \in \setN$ that
	\begin{align}\label{eq:j4} 
		\lim_{n\to\infty} J_4^{n,j} +J_7^{n,j}  =0\,.
	\end{align}
	From \eqref{eq:diff1}--\eqref{eq:j4}, it follows \eqref{mon4}. Thus, Corollary \ref{cor:ae-conv} yields subsequences~with
	\begin{align}
          \begin{aligned}
		\nabla\bv^n&\to\nabla\bv&&\quad\text{ a.e.~in }\Omega\,,\\[-0.5mm] 
		\hat\nabla\w^n&\to\hat\nabla\w&&\quad\text{ a.e.~in }\Omega\,,\\
		\w^n&\to\,\w&&\quad\text{ a.e.~in }\Omega\,.
              \end{aligned}\label{pw}
	\end{align}
	 Since $\mathbf{S}\in C^0(\mathbb{R}^{d\times d}_{\textup{sym}}\times \mathbb{R}^{d\times d}_{\textup{skew}}\times \mathbb{R}^d;\mathbb{R}^{d\times d})$~(cf.~(\hyperlink{(S.1)}{S.1})) and $\mathbf{N}\in C^0(\mathbb{R}^{d\times d}\times \mathbb{R}^d;\mathbb{R}^{d\times d})$ (cf.~(\hyperlink{(N.1)}{N.1})), we
	deduce from \eqref{pw} that
	\begin{align}\label{eq:aen}
		\begin{aligned}
			\bS^n&\to \Ss& &\text{a.e.~in}\ \Omega&&\quad(n\to \infty)\,,
			\\
			\bN^n&\to  \bN(\hat\nabla\w,\bE)& &\text{a.e.~in}\ \Omega&&\quad(n\to \infty)\,.
		\end{aligned}
	\end{align}
	To identify $\widehat \bS$, we now argue as in the proof of
	\mbox{\cite[Thm.~4.6  \!(cf.  \!(4.21)$_1$--(4.23)$_1$)]{erw}}, while Theorem
	\ref{pfastue} (with $G=\Omega$ and $\sigma=\smash{\abs{\bE}^2}$),
	\eqref{konvergenz18}, \eqref{eq:aen} and the absolute continuity of
	Lebesgue measure with respect to 
	$\nu_{\smash{\abs{\bE}^2}}$ is used to
	identify~$\smash{\widehat \bN}$. Thus, we just proved
	\begin{align}
		\begin{aligned}\label{eq:SN}
			\widehat \bS=\Ss \quad \text {and} \quad 
			\widehat \bN=\bN(\hat\nabla\w,\bE)\,.
		\end{aligned}
	\end{align}
	\textbf{4. Limiting process $n\to \infty$:}\\
	Now we have at our disposal everything to identify the limits of all
	but~one~term in \eqref{NS2}. Using \eqref{apri},  \eqref{konvergenz}, \eqref{konvergenz18},
	\eqref{conv-convec}$_1$,  \eqref{eq:SN} as well as  $p> \frac {2d}{d+2}$, we~obtain from \eqref{NS2}
	that  for every
	$\boldsymbol\varphi \in C^1_0(\Omega)$ with $\divo
	\bphi =0$ and for every $\boldsymbol\psi \in C_0^1(\Omega)$,~it~holds
	\begin{align}\begin{aligned}
		&\big\langle \Ss-\bv\otimes\bv,\bD\bphi+\bR(\bphi,\boldsymbol\psi)\big\rangle 
		\\
		& \;+\big\langle\bN(\hat\nabla\w,\bE),\nabla\boldsymbol\psi\big\rangle -\lim _{n\to
			\infty
		}\big\langle\w^n\otimes\bv^n,\nabla\boldsymbol\psi\big\rangle
		=\big\langle
		\ff,\boldsymbol\varphi\big\rangle+\big\langle\bell,\boldsymbol\psi\big\rangle
		\,.
	\end{aligned}\label{NS2aa}
	\end{align}
	Finally, we have to identify 
        the remaining
        limit~in~\eqref{NS2aa}. 
        To this end,  we
        fix an arbitrary $\boldsymbol\psi \in C_0^1(\Omega)$ with
        $\nabla \boldsymbol\psi\in
        L^{\smash{\frac{q}{q-2}}}(\Omega;\vert
        \bfE\vert^{\smash{-\frac{\alpha q}{q-2}}})$ and choose
        $\Omega'$ with Lipschitz boundary such that 
        $\textup{int}(\textup{supp}(\boldsymbol\psi)) \subset\subset
        \Omega' \subset\subset\Omega$ holds. Due
        to Theorem~\ref{compactnew} and
        $\eqref{konvergenz}_3$,~for~all~${q\hspace*{-0.15em}\in\hspace*{-0.15em}
          \left[1,p^*\right)}$,~it~holds
	\begin{align*}
		\w^n\rightharpoonup \w\quad\text{ in }L^q(\Omega';\vert\bE\vert^{\alpha q})\quad(n\to \infty)
	\end{align*}
	for every $\alpha\hspace*{-0.1em}\ge \hspace*{-0.1em}1+\frac{2}{p}$. On the other hand, due to $\nabla \boldsymbol\psi\hspace*{-0.1em}\in\hspace*{-0.1em} L^{\smash{\frac{q}{q-2}}}(\Omega;\vert \bfE\vert^{\smash{-\frac{\alpha q}{q-2}}})$~and~$\eqref{konvergenz}_2$, using H\"older's inequality, for any $q\in \left[1,p^*\right)$, we also see that
	\begin{align*}
		\nabla\boldsymbol\psi\,\bv^n\rightarrow\nabla\boldsymbol\psi\,\bv\quad\text{ in }L^{q'}(\Omega';\vert\bE\vert^{\frac{-\alpha q}{q-1}})\quad(n\to \infty)\,.
	\end{align*}
	Since $(L^q(\Omega';\vert\bE\vert^{\alpha
          q}))^*\hspace*{-0.1em}=\hspace*{-0.1em}L^{q'}(\Omega';\vert\bE\vert^{\frac{-\alpha
            q}{q-1}})$, we infer that
        $$
        {\lim _{n\to \infty
          }\big\langle\w^n\otimes\bv^n,\nabla\boldsymbol\psi\big\rangle}
        \big\langle\w\otimes\bv,\nabla\boldsymbol\psi\big\rangle\,,
        $$
        which, looking back to \eqref{NS2aa}, concludes the proof of
        Theorem~\ref{thm:main4}.
\end{proof}

\section{Variable shear exponent}\label{veroeffentlichung1p(x)}

In this section, we extend the existence result in Theorem
\ref{thm:main4}~to~the~case~of variable exponents. Before we do so, we first give a
brief introduction~into~weighted variable exponent Lebesgue and
Sobolev spaces. Then, we explain the changes~in the arguments in the
previous sections due to the variable~exponent~setting. 
\subsection{Weighted variable exponent Lebesgue and Sobolev spaces}

Let $\Omega\subseteq \mathbb{R}^d$, $d\in \mathbb{N}$, be an open set
and $p:\Omega\to [1,\infty)$ be a measurable function, called variable
exponent. By $\mathcal{P}(\Omega)$, we denote the set of all variable
exponent. For $p\in \mathcal{P}(\Omega)$, we denote by
${p^+\!:=\textup{ess\,sup}_{x\in
    \Omega}{p(x)}}$~and~${p^-\!:=\textup{ess\,inf}_{x\in
    \Omega}{p(x)}}$ its constant limit exponents. By
$\mathcal{P}^{\infty}(\Omega):=\{p\in\mathcal{P}(\Omega)\mid
p^+<\infty\}$, we denote the set of all bounded variable exponents. For
$p\in\mathcal{P}(\Omega)$,~we~use~the,~by~now~\mbox{standard}, variable exponent 
Lebesgue spaces $L^{p(\cdot)}(\Omega)$ equipped with the Luxembourg
norm $\|\cdot\|_{p(\cdot)}$ and  Sobolev spaces
$W^{1,p(\cdot)}(\Omega)$ with the norm
$\|\cdot\|_{1,p(\cdot)}:=\|\cdot\|_{p(\cdot)}
+\|\nabla\cdot\|_{p(\cdot)}$. These spaces are separable Banach
spaces.~The~space~$W^{1,p(\cdot)}_0(\Omega)$~is~defined as the
completion of $C_0^\infty(\Omega)$ with respect to the gradient norm
$\Vert\nabla \cdot\Vert_{p(\cdot)}$, while the space $V_{p(\cdot)}(\Omega)$ is
the closure of $C_{0,\textup{div}}^\infty(\Omega)$ with respect to the
gradient~norm~$\|\nabla \cdot\|_{p(\cdot)}$. By
$\smash{L^{p(\cdot)}_0(\Omega)}$, we denote the subspace of
$L^{p(\cdot)}(\Omega)$ consisting of all functions with
vanishing~mean~value. If $p\in \mathcal{P}^{\infty}(\Omega)$, in
addition, satisfies $p^->1$, then~the~spaces
$L^{p(\cdot)}(\Omega)$, $L^{p(\cdot)}_0(\Omega)$,
$\smash{W^{1,p(\cdot)}(\Omega)}$,
$\smash{W^{1,p(\cdot)}_0(\Omega)}$ and $\smash{V_{p(\cdot)}(\Omega)}$
are reflexive.  For a~more in-depth  analysis of these spaces, we refer to
\cite{KR91}, \cite{fan-zhao},~\cite{lpx-book}~and~\cite{CUF13}.

For a variable exponent $p\in \mathcal{P}^{\infty}(\Omega)$ and a
weight $\sigma \in L^1_\loc(\setR^d)$, the weighted variable exponent Lebesgue space
$L^{p(\cdot)}(\Omega;\sigma)$ consists of all measurable~functions
${u:\Omega\to \mathbb{R}}$, i.e., $u\in \mathcal{M}(\Omega)$, for
which the modular
 \begin{align*}
 	\rho_{p(\cdot),\sigma}(u):=\int_{\Omega}{\vert u(x)\vert^{p(x)}\,d\nu_\sigma(x)}:=\int_{\Omega}{\vert u(x)\vert^{p(x)}\sigma(x)\,dx}
 \end{align*}
 is finite, i.e., we have that
 $L^{p(\cdot)}(\Omega;\sigma):=\{u\in \mathcal{M}(\Omega)\mid
 \sigma^{1/p(\cdot)}u\in L^{p(\cdot)}(\Omega)\}$.~Then, we equip
 $L^{p(\cdot)}(\Omega;\sigma)$ with the Luxembourg norm
\begin{align*}
	\|u\|_{p(\cdot),\sigma}:=\inf\big\{\lambda > 0\mid \rho_{p(\cdot),\sigma}(u/\lambda)\leq 1\big\}\,,
\end{align*}
which turns $L^{p(\cdot)}(\Omega;\sigma)$ into a separable Banach
space. If $p\in \mathcal{P}^{\infty}(\Omega)$, in addition, satisfies
$p^->1$, then $L^{p(\cdot)}(\Omega;\sigma)$ is reflexive. The dual
space $(L^{p(\cdot)}(\Omega;\sigma))^*$ can be identified with
respect~to~$\skp{\cdot}{\cdot}$~with~$ L^{p'(\cdot)}(\Omega;\sigma')$,
where $\sigma':=\sigma^{\smash{\frac{-1}{p(\cdot)-1}}}$. These
properties, as many other basic properties of weighted variable
Lebesgue spaces, can be proved in the same way as for variable
Lebesgue spaces. This observation works for all results for which no
particular property of the Lebesgue measure is used that is not
shared by a Radon measure $\nu_\sigma$ (cf.~\cite{rubo}, \cite{lpx-book}).

The identity $\rho_{p(\cdot),\sigma}(u)=\rho_{p(\cdot)}(u\sigma^{1/p(\cdot)})$
implies that
\begin{gather*}
  \|u\|_{p(\cdot),\sigma}=\|u\sigma^{1/p(\cdot)}\|_{p(\cdot)}
\end{gather*}
for all $u\in L^{p(\cdot)}(\Omega;\sigma)$. This and H\"older's
inequality in variable Lebesgue spaces, for every
$u\in L^{p(\cdot)}(\Omega;\sigma)$ and
$v \in L^{p'(\cdot)}(\Omega;\sigma')$, where  $\sigma'=\sigma^{\smash{\frac{-1}{p(\cdot)-1}}}$, yields that
\begin{equation*}
  \bigabs{\skp{u}{v}}\le 2\,\norm{u}_{p(\cdot),\sigma} \norm{v}_{p'(\cdot),\sigma'}\,.
\end{equation*}

The relation between the modular and the norm is clarified by the
following lemma, which is called norm-modular unit ball property.
\begin{Lem}
  \label{lem:unit_ball_px}
  Let $\Omega\subseteq \R^d$, ${d\in
    \mathbb{N}}$,~be~open and let $p\in
  \mathcal{P}^{\infty}(\Omega)$. Then, we have for
  any $u \in \Lpwx$:
  \begin{enumerate}[(i)]
    \item $\norm{u}_{p(\cdot),\sigma} \leq 1$ if and only if   $\rho_{p(\cdot),\sigma} (u) \leq 1$. 
    \item \label{itm:unit_ball2pxa} If $\norm{u}_{p(\cdot),\sigma} \leq 1$, then
      $\rho _{p(\cdot),\sigma} (u) \leq
      \norm{u}_{p(\cdot),\sigma}$.\label{itm:unit_ball2pxb}
    \item If $ 1< \norm{u}_{p(\cdot),\sigma}$, then
      $\norm{u}_{p(\cdot),\sigma} \leq \rho _{p(\cdot),\sigma} (u)$.
    \item $\smash{\norm{u}_{p(\cdot),\sigma}^{p^-} -1 \le \rho _{p(\cdot),\sigma} (u) \leq
      \norm{u}_{p(\cdot),\sigma}^{p^+} +1}$.
  \end{enumerate}
\end{Lem}
\begin{Bew}
	See \cite[Lem.~3.2.4 \& Lem.~3.2.5]{lpx-book}.
\end{Bew}

In order to define weighted variable exponent Sobolev spaces, in analogy with Assumption \ref{weight}, we make the following assumption.

\begin{Vss}\label{weightp(x)}
  Let $\Omega\subseteq\R^d$, $d\in \mathbb{N}$, be an open set and
  $p\in\mathcal{P}^{\infty}(\Omega)$.  The weight $\sigma$ is admissible,
  i.e., if a sequence ${(\varphi_n)_{n\in \mathbb{N}}\subseteq C^\infty(\Omega)}$ and
  ${\bv \in L^{p(\cdot)}(\Omega;\sigma)}$ satisfy
  ${\int_\Omega\vert\varphi_n(x)\vert^{p(x)}\sigma(x)\,dx\!\to\!  0}$
  $(n\!\to\! \infty)$ and
  ${\int_\Omega\vert\nabla\varphi_n(x)\!-\!\bv(x)\vert^{p(x)}\sigma(x)\,dx\!\to\!
     0}$ $(n\to\infty)$, then it follows that $\bv=\mathbf{0} $ in
  $L^{p(\cdot)}(\Omega;\sigma)$.
\end{Vss}
\begin{Bem}
		If $\sigma \in C^0(\Omega)$, then the same argumentation as in Remark~\ref{weightexamples}~(ii) shows that  Assumption~\ref{weightp(x)} is satisfied for every $p\in \mathcal{P}^{\infty}(\Omega)$.
\end{Bem}

For $\sigma$ satisfying Assumption~\ref{weightp(x)} 
and $p \in \mathcal{P}^{\infty}(\Omega)$, we introduce the norm
$$
\Vert u\Vert_{1,p(\cdot),\sigma}:=
\|u\|_{p(\cdot),\sigma}
+\|\nabla u\|_{p(\cdot),\sigma}\,,
$$ 
whenever the right-hand side is well-defined.  

\begin{Def}
	Let  $\Omega\subseteq\R^d$, $d\in \mathbb{N}$, be open and let Assumption~\ref{weightp(x)}~be~satisfied. Then, the weighted variable exponent Sobolev space
	$H^{1,p(\cdot)}(\Omega;\sigma)$ is defined as the completion of
	$\mathcal{V}_{p(\cdot),\sigma}\hspace*{-0.1em}:=\hspace*{-0.1em}\{u\hspace*{-0.1em}\in\hspace*{-0.1em} C^\infty(\Omega)\fdg
	\Vert u\Vert_{1,p(\cdot),\sigma}\hspace*{-0.1em}<\hspace*{-0.1em}\infty\}$ with~respect~to~${\Vert\cdot\Vert_{1,p(\cdot),\sigma}}$.
\end{Def}

In other words, $u \in H^{1,p(\cdot)}(\Omega;\sigma)$ if and
only if $u\in L^{p(\cdot)}(\Omega;\sigma)$ and there exists a function $\bv\in L^{p(\cdot)}(\Omega;\sigma)$ such that for
some sequence
$(\varphi_n)_{n\in \mathbb{N}}\subseteq C^\infty(\Omega)$ holds both
$\Io\vert\varphi_n-u\vert^{p(x)}\sigma\,dx\to 0$ $(n\to\infty)$ and
${\Io\vert\nabla\varphi_n-\bv\vert^{p(x)}\sigma\,dx\to 0}$~${(n\to\infty)}$. 
Assumption \ref{weightp(x)} implies that
$\bv $ is a uniquely defined function in $L^{p(\cdot)}(\Omega;\sigma)$~and~we, thus,  define
$\hat{\nabla}u:=\bv$. Note that $W^{1,p(\cdot)}(\Omega)=H^{1,p(\cdot)}(\Omega;\sigma)$ if $\sigma= 1$ a.e.~in $\Omega$ with ${\nabla u =\hat\nabla u}$~for~all~${u\in W^{1,p(\cdot)}(\Omega)}$. However, in
general, $\hat{\nabla}u$ and~the~usual~weak or distributional gradient
$\nabla u$~do~not~coincide. 
 Then, the space $H^{1,p(\cdot)}_0(\Omega;\sigma)$ is defined as the closure
 of $C_0^\infty(\Omega)$ with respect to the
 $\|\cdot\|_{1,p(\cdot),\sigma}$--norm.~If~${\sigma\in L^\infty(\Omega)}$, then
 $\smash{W^{1,p(\cdot)}_0(\Omega)\hookrightarrow
 H^{1,p(\cdot)}_0(\Omega;\sigma)}$ and ${\nabla u =\hat\nabla u}$ for
 all $\smash{u\in W^{1,p(\cdot)}_0(\Omega)}$, which is a consequence of
 \begin{gather*}
   \|v\|_{p(\cdot),\sigma}=\|v\sigma^{1/p(\cdot)}\|_{{p(\cdot)}}\leq 2\,
   \|\sigma\|_{\infty}^{1/p^-}\|v\|_{{p(\cdot)}} 
 \end{gather*}
 valid for every $v\in L^{p(\cdot)}(\Omega)$.

Another possible approach is to define the weighted variable Sobolev space $\Wpwx$ as
the set of all functions $u\in \Lpwx$ which posses a distributional gradient
${\nabla u \in\Lpwx}$. We equip $\Wpwx$ with the norm
$\Vert\cdot\Vert_{1,p,\sigma}$. As constant exponents are a particular
case we have that, in general, the space $\Wpwx$ need not to
be a Banach space (cf.~\cite{heinonen}). The space $\Wpwx$ is mostly
studied in the particular case that ${\sigma
^{\smash{\frac {-1}{p(\cdot)-1}}} \in L^1_\loc (\Omega)}$, which ensures that
$\Wpwx$ is a Banach space and that $\smash{\nabla u =\hat \nabla u}$
(cf.~\cite{KWZ10}, \cite{heinonen}). However, this
condition is again for our purposes too restrictive (cf.~Section
\ref{sec:E}). Thus, we~will~not~use~$\Wpwx$, but we will work with the
spaces $\Hpwx$.~Since~the~space~$\Hpwx$~is~even less studied (we are
only aware of~the~study~in~\cite{Su14}),~we~prove~its~basic~properties.
\begin{Sa}
  Let $\Omega\subseteq \mathbb{R}^d$, $d\in \mathbb{N}$, be an open
  set and let $p\in \mathcal{P}^{\infty}(\Omega)$ satisfy $p^->1$. Then, the space
  $H^{1,p(\cdot)}(\Omega;\sigma)$ is a separable and reflexive Banach
  space.
\end{Sa}

\begin{proof}
  The space $H^{1,p(\cdot)}(\Omega;\sigma)$, by
  definition, is a Banach space. So, it is left to check that it is separable and
  reflexive. For this, we first note that
  \begin{align}
    \|u\|_{1,p(\cdot),\sigma}=\|u\|_{p(\cdot),\sigma}+\|\hat\nabla u\|_{p(\cdot),\sigma}\label{isometry}
  \end{align}
  for all $u\in H^{1,p(\cdot)}(\Omega;\sigma)$. In fact, for any
  $u\in H^{1,p(\cdot)}(\Omega;\sigma)$, by definition, there exists a
  sequence
  $(\varphi_n)_{n\in
    \mathbb{N}}\subseteq\mathcal{V}_{p(\cdot),\sigma}$ such that
  $\varphi_n\to u$ in $L^{p(\cdot)}(\Omega;\sigma)$ $(n\to \infty)$,
  $\nabla\varphi_n\to \hat\nabla u$ in $L^{p(\cdot)}(\Omega;\sigma)$
  $(n\to \infty)$ and
  $\|\varphi_n\|_{1,p(\cdot),\sigma}= \|\varphi_n\|_{p(\cdot),\sigma}
  +\|\nabla\varphi_n\|_{p(\cdot),\sigma}$, $n\in \mathbb{N}$. Thus, by
  passing for $n\to \infty$, we
  obtain~\eqref{isometry}~for~all~${u\in
    H^{1,p(\cdot)}(\Omega;\sigma)}$.  The equality \eqref{isometry} in
  turn implies that
  ${\Pi: H^{1,p(\cdot)}(\Omega;\sigma)\to
    L^{p(\cdot)}(\Omega;\sigma)^{d+1}}$, defined via
  $\Pi u:=(u,\hat\nabla u)^\top $ in
  $L^{p(\cdot)}(\Omega;\sigma)^{d+1}$ for every
  $u\in H^{1,p(\cdot)}(\Omega;\sigma)$, is an isometry. In particular,
  $\Pi$ is an isometric isomorphism from $H^{1,p(\cdot)}(\Omega;\sigma)$~onto~its~range~$R(\Pi)$. 
  Thus, $R(\Pi)$ inherits the separability and reflexivity of
  $L^{p(\cdot)}(\Omega;\sigma)^{d+1}$ and, by virtue of the
  isometric~isomorphism,~$H^{1,p(\cdot)}(\Omega;\sigma)$~as well.
 \end{proof}


\subsection{$\log$--H\"older continuity and related results}\label{sec:log}

We say that a bounded exponent $p\in \mathcal P^\infty (G)$ is locally
$\log$--Hölder continuous, if there is a constant $c_1>0$ such that
for all $x,y\in G$
\begin{align*}
	\vert p(x)-p(y)\vert \leq \frac{c_1}{\log(e+1/\vert x-y\vert)}\,.
\end{align*}
We say that $p \in \mathcal P^\infty (G)$ satisfies the $\log$--Hölder decay condition, if there exist 
constants $c_2>0$ and $p_\infty\in \setR$ such that for all $x\in G$
\begin{align*}
	\vert p(x)-p_\infty\vert \leq\frac{c_2}{\log(e+1/\vert x\vert)}\,.
\end{align*} 
The exponent $p$ is called globally $\log$--Hölder continuous on $G$, if it is locally 
$\log$--Hölder continuous and satisfies the $\log$--Hölder decay condition. 
The maximum $c_{\log}(p):=\max\{c_1,c_2\}$ is just called the $\log$--Hölder constant of $p$.
Furthermore, we denote by $\mathcal{P}^{\log}(G)$ the set of 
globally $\log$--Hölder continuous
functions on $G$. 

$\log$--Hölder continuity  is a special modulus of continuity for variable exponents that is sufficient for the validity of the following results.

\begin{Sa}\label{bogp(x)}
  Let $G\subseteq \mathbb{R}^d\!$, $d\ge 2$, be a bounded Lipschitz
  domain. Then, there exists a linear operator
  $\mathcal{B}_G:C^\infty_{0,0}(G)\to C^{\infty}_{0}(G)$ which for all
  exponents $\smash {p\in \mathcal{P}^{\log}(G)}$ satisfying $p^->1$
  extends uniquely to a linear, bounded operator
  ${\mathcal{B}_G:L^{p(\cdot)}_0(G)\to W^{1,p(\cdot)}_0(G)}$ such that
  $\|\mathcal{B}_Gu\|_{1,p(\cdot)}\leq c\,\|u\|_{p(\cdot)}$ and
  $\textup{div}\,\mathcal{B}_Gu = u$ for every
  $\smash{u\in L^{p(\cdot)}_0(G)}$.
\end{Sa}
\begin{Bew}
  See \cite[Thm.~2.2]{dr-nta}, \cite[Thm.~6.4]{dr-calderon}, \cite[Thm.~14.3.15]{lpx-book}.
\end{Bew}

\begin{Sa}\label{poincarep(x)}
  Let $G\subseteq \mathbb{R}^d$, $d\in \mathbb{N}$, be a bounded
  Lipschitz domain~and let $\smash{p \in \mathcal{P}^{\log}(G)}$ satisfy
  $p^->1$.  Then, there exists a constant $c >0$ such that
  ${\| \bfu\|_{p(\cdot)}\leq c\, \|\nabla \bfu\|_{p(\cdot)}}$
	for~every~$\smash{\bfu\in W^{1,p(\cdot)}_0(G)}$.
\end{Sa}
\begin{Bew}
	See \cite[Thm.~8.2.4]{lpx-book}.
\end{Bew}

\begin{Sa}\label{kornp(x)}
  Let $G\subseteq \mathbb{R}^d\!$, $d\in \mathbb{N}$, be a bounded
  Lipschitz domain and let $\smash{p\in\mathcal{P}^{\log}(G)}$ satisfy
    $p^- > 1$.  Then, there exists a constant $c>0$ such that
  ${ \|\nabla \bfu\|_{p(\cdot)}\leq c\, \|\bfD \bfu\|_{p(\cdot)}}$
	for~every~$\smash{\bfu\in W^{1,p(\cdot)}_0(G)}$.
\end{Sa}
\begin{Bew}
	See \cite[Thm.~5.5]{dr-calderon}, \cite[Thm.~14.3.21]{lpx-book}.
\end{Bew}

\begin{Sa}\label{thm:Ltp(x)}Let $G\subseteq \mathbb{R}^d$,
$d\in\mathbb{N}$, be a bounded
Lipschitz~domain,~${p\in\mathcal{P}^{\log}(G)}$ with $p^->1$ and
	let $\bfu^n \in W^{1,p(\cdot)}_0(G)$ be such that
	$\bfu^n \weakto \bfzero$ in $W^{1,p(\cdot)}_0(G)$ $(n\to \infty)$.  Then,
	for any ${j, n\hspace*{-0.1em}\in\hspace*{-0.1em} \setN}$, there
	exist
	$\bfu^{n,j}\hspace*{-0.2em} \in\hspace*{-0.1em} W^{1,\infty}_0(G)$
	and~$\smash{\lambda_{n,j}\hspace*{-0.1em} \in \hspace*{-0.1em}\big
		[2^{2^j}\hspace*{-0.1em}, 2^{2^{j+1}}\big ]} $~such~that
	\begin{align*}
			\smash{ \lim_{n\to \infty}} \big ( {\sup}_{j \in \setN}
			\norm{\bfu^{n,j}}_{\infty}\big ) &=0\,,\\
			\norm{\nabla \bfu^{n,j}}_{\infty} &\leq c\, \lambda_{n,j}\leq c\,
			\smash{2^{2^{j+1}}}\,,
			\\
			\bignorm{\nabla \bfu^{n,j}\, \chi_{
					\set{\bfu^{n,j} \not= \bfu^n}}}_{p(\cdot)} &\leq c\, \big\| \lambda_{n,j} \chi_{\set{\bfu^{n,j} \not= \bfu^n}}\big\|_{p(\cdot)} \,,
			\\[-1mm]
			\smash{\limsup _{n \to \infty} }\,\big\| \lambda_{n,j} \chi_{\set{\bfu^{n,j} \not= \bfu^n}}\big\|_{p(\cdot)} &\leq c\, 2^{-j/p^+}\,,
	\end{align*}
	where $c=c(d,p,G)>0$.  Moreover, for any $j \in
	\setN$, $\nabla \bfu^{n,j} \weakto \bfzero $~in~$L^s(G) $~${(n \to
		\infty)}$,
	$s \in [1,\infty)$, and $\nabla \bfu^{n,j} \stackrel{*}{\weakto}
	\bfzero$ in $L^\infty(G)$ $(n \to
	\infty)$. 
\end{Sa}
\begin{Bew}
	See \cite[Thm.~4.4]{dms}, \cite[Cor.~9.5.2]{lpx-book}.
\end{Bew}
$\log$--H\"older continuity is also sufficient to prove the analogue of
Lemma \ref{hatgrad} in the variable exponent case.
\begin{Lem}\label{hatgradp(x)}
	Let $\Omega\subseteq \R^d$,
	${d\in \mathbb{N}}$,~be~open,~$p\in \mathcal{P}^{\log}(\Omega)$ and let
	Assumption~\ref{VssE} be satisfied.  Then,
	for any $\Omega'\subset\subset\Omega_0$, we have that
	${W^{1,p(\cdot)}(\Omega')=H^{1,p(\cdot)}(\Omega';\vert\bfE\vert^2)}$ with norm
	equivalence (depending on $\Omega'\!$ and $\bfE$) and
	$\hat{\nabla} u\!=\!\nabla u$ for all ${u\!\in\! W^{1,p(\cdot)}(\Omega')}$. 
\end{Lem}

\begin{proof}
  Due to $\vert \bE\vert >0 $ in $\overline{\Omega'}$ and
  ${\vert \bE\vert\in C^0(\overline{\Omega'})}$, there is a local
  constant $c(\Omega')\!>\!0$ such that
  $c(\Omega')^{-1}\leq \vert \bE\vert^2\leq
  c(\Omega')$~in~$\overline{\Omega'}$. Thus,
  ${L^{p(\cdot)}(\Omega')=L^{p(\cdot)}(\Omega';\vert
    \bE\vert^2)}$~with
  \begin{align*}
    \smash{c(\Omega')^{-\frac{1}{p^-}}\|u\|_{L^{p(\cdot)}(\Omega')}\leq
    \|u\|_{L^{p(\cdot)}(\Omega';\vert \bE\vert^2)}\leq
    c(\Omega')^{\frac{1}{p^-}}\|u\|_{L^{p(\cdot)}(\Omega')}} 
  \end{align*}
  for all
  $u\in L^{p(\cdot)}(\Omega')=L^{p(\cdot)}(\Omega';\vert
  \bE\vert^2)$. As a result, it holds~$\mathcal{V}_{p(\cdot),\vert
    \bE\vert^2}=\mathcal{V}_{p(\cdot),1}$~with
  \begin{align}
    \smash{c(\Omega')^{-\frac{1}{p^-}}\|u\|_{W^{1,p(\cdot)}(\Omega')}\leq
    \|u\|_{H^{1,p(\cdot)}(\Omega';\vert \bE\vert^2)}\leq
    c(\Omega')^{\frac{1}{p^-}}\|u\|_{W^{1,p(\cdot)}(\Omega')}}\label{equivalencep(x)} 
  \end{align}
  for all  $u\! \in\! \mathcal{V}_{p(\cdot),\vert
    \bE\vert^2}\!=\!\mathcal{V}_{p(\cdot),1}$. Since
  $W^{1,p(\cdot)}(\Omega')$, by \cite[\!Thm. \!9.1.8.]{lpx-book},~is~the~closure~of $\mathcal{V}_{p(\cdot),1}$ and
  $H^{1,p(\cdot)}(\Omega';\vert\bfE\vert^2)$, by definition,~is~the~closure~of
  $\mathcal{V}_{p(\cdot),\vert \bE\vert^2}$,~
  \eqref{equivalencep(x)}~implies that
  ${W^{1,p(\cdot)}(\Omega')=H^{1,p(\cdot)}(\Omega';\vert\bfE\vert^2)}$
  and $\smash{\hat{\nabla}} u=\nabla u$ for all
  ${u\in W^{1,p(\cdot)}(\Omega')}$.
\end{proof}

\subsection{A  weak stability lemma for variable exponents}\label{sec:weakp(x)}
Also the weak stability of problems of $p(\cdot)$--Laplace type~is~well~known~(cf.~\cite{dms}). It also holds for our problem \eqref{NS} if we make
appropriate natural assumptions on the extra stress tensor $\mathbf{S}$ and on the couple stress tensor
$\mathbf{N}$, which are motivated by the canonical example in~\eqref{eq:SN-ex}.

\begin{Vss}\label{VssSpx}
	For the extra stress tensor
	$\smash{\mathbf{S}:\R_{\sym}^{d\times d}\times \R_{\anti}^{d\times
         d}\times\R^d\to \setR^{d}}$  and some $\smash{\hat p \in \mathcal
       P^{\log}(\setR)}$ with $\smash{\hat p^- >1}$, there
       exist constants $c,C >0$ such that: 
	\begin{enumerate}
		\item[{\rm \hypertarget{(S.1)}{(S.1)}}] $\smash{\mathbf{S}\in C^0(\R_{\sym}^{d\times d}\times \R_{\anti}^{d\times d}\times \R^d;\setR^{d
			\times d})}$.
		\item[{\rm \hypertarget{(S.2)}{(S.2)}}] For every $\smash{\bD \in \R_{\sym}^{d\times d}}$,
		$\smash{\bR \in \R_{\anti}^{d\times d}}$ and $\bE \in \setR^d$, it holds\\[-5mm]
		\begin{align*}
				\vert\mathbf{S}^{\sym}(\mathbf{D},\mathbf{R},\mathbf{E})\vert&\leq
				c\,\big (1+\vert\E\vert^2 \big) \big (1+\vert\mathbf{D}\vert^{\hat p(\abs{\bE}^2)-1}\big)\,,
				\\[-1mm]
				\vert\mathbf{S}^{\anti}(\mathbf{D},\mathbf{R},\mathbf{E})\vert&\leq
				c\,\vert \mathbf{E}\vert^2 \big (1+\vert\mathbf{R}\vert^{\hat p(\abs{\bE}^2)-1}\big)\,.
		\end{align*}\\[-7mm]
		
		\item[{\rm \hypertarget{(S.3)}{(S.3)}}]  For every $\smash{\bD \in \R_{\sym}^{d\times d}}$,
		$\smash{\bR \in \R_{\anti}^{d\times d}}$ and $\bE \in \setR^d$, it holds\\[-5mm]
		\begin{align*}
				\mathbf{S}(\mathbf{D},\mathbf{R},\mathbf{E}):\mathbf{D}
				&\geq  c\,\big (1+\vert\E\vert^2\big )\, \big( \vert\mathbf{D}\vert^{\hat p(\abs{\bE}^2)}-C\big)\,,
				\\[-1mm]
				\mathbf{S}(\mathbf{D},\mathbf{R},\mathbf{E}):\mathbf{R}&\geq
				c\,\vert \mathbf{E}\vert^2 \big( \vert\mathbf{R}\vert^{\hat p(\abs{\bE}^2)}-C\big)\,.
		\end{align*}\\[-7mm]
		
		\item[{\rm \hypertarget{(S.4)}{(S.4)}}]  For every $\smash{\bD_1, \bD_2 \in
		\R_{\sym}^{d\times d}}$, $\smash{\bR_1, \bR_2 \in \R_{\anti}^{d\times d}}$
		and $\bE \in \setR^d$ with $(\mathbf{D}_1,\vert
		\mathbf{E}\vert \mathbf{R}_1)\neq(\mathbf{D}_2,\vert
		\mathbf{E}\vert \mathbf{R}_2)$, it holds\\[-6mm]
		\begin{align*}
				\big (\mathbf{S}(&\mathbf{D}_1,\mathbf{R}_1,\mathbf{E})-
				\mathbf{S}(\mathbf{D}_2,\mathbf{R}_2,\mathbf{E})\big ):
				\big (\mathbf{D}_1-\mathbf{D}_2+\mathbf{R}_1-\mathbf{R}_2\big )>0\,.
		\end{align*}
	\end{enumerate}
\end{Vss}

\begin{Vss}\label{VssNpx}
	For the couple stress tensor
	$\smash{\mathbf{N}:\R^{d\times d}\times \R^d\to\R^{d\times d}}$ and some $\smash{\hat p \in \mathcal
       P^{\log}(\setR)}$ with $\smash{ \hat p^- >1}$, there
       exist constants $c, C >0$ such that:  
	\begin{enumerate}
		\item[{\rm \hypertarget{(N.1)}{(N.1)}}]$\smash{\mathbf{N}\in C^0(\R^{d\times d}\times \R^d;\R^{d\times d})}$.
		
		\item[{\rm \hypertarget{(N.2)}{(N.2)}}]  For every $\smash{\bL \in \R^{d\times d}}$ and $\bE \in
		\setR^d$, it holds\\[-6mm]
		\begin{align*}
		\smash{	\vert\mathbf{N}(\mathbf{L},\mathbf{E})\vert\leq c\,\big
			\vert\mathbf{E}\vert^2
			\big(1+\vert\mathbf{L}\vert^{\hat p(\abs{\bE}^2)-1}\big )\,.}
		\end{align*}\\[-10mm]
		
		\item[{\rm \hypertarget{(N.3)}{(N.3)}}] For every $\bL \in \R^{d\times d}$ and $\bE \in
		\setR^d$, it holds\\[-6mm]
		\begin{align*}\label{Nkoerzivpx}
		\smash{	\mathbf{N}(\mathbf{L},\mathbf{E}):\mathbf{L}\geq
			c\,\big \vert\mathbf{E}\vert^2
			\big(\vert\mathbf{L}\vert^{\hat p(\abs{\bE}^2)} -C\big)\,.}
		\end{align*}\\[-10mm]

             \item[{\rm \hypertarget{(N.4)}{(N.4)}}]  For every $\bL_1,\bL_2 \in \R^{d\times d}$ and $\bE \in
		\setR^d$ with $\vert \mathbf{E}\vert>0$ and
		$\mathbf{L}_1\neq \mathbf{L}_2$, it holds\\[-6mm]
		\begin{align*}
			(\mathbf{N}(\mathbf{L}_1,\mathbf{E})-\mathbf{N}
			(\mathbf{L}_2,\mathbf{E})):(\mathbf{L}_1-\mathbf{L}_2)>0\,.
		\end{align*}
	\end{enumerate}
\end{Vss}

Concerning the material function $\hat p$ in Assumption \ref{VssSpx}
and Assumption \ref{VssNpx}, we assume the following: 

\begin{Vss}\label{VssEp(x)}
  Let Assumption~\ref{VssE} be satisfied and let
  $\hat p \in \mathcal P^{\log}(\setR) 
$. Then, the exponent $p:\Omega\to [1,\infty)$, defined
via
	\begin{align*}
		p(x):=\hat p(\vert \bfE(x)\vert^2)
	\end{align*} 
	for every $x\in \Omega$, satisfies
	$p\in \mathcal{P}^{\log}(\Omega)$.
\end{Vss}

\begin{Bem}
	Assumption \ref{VssEp(x)} can be verified under certain
	conditions on the boundary data $\bE_0$. 
        In fact, the
	regularity~theory~of~Maxwell's equations (cf.~\!\cite{gsch},
	\cite{rubo}) implies $\bE \hspace*{-0.15em}\in\hspace*{-0.15em} C^{0,\alpha}(\overline {\Omega})$,
	$\alpha\hspace*{-0.15em} \in\hspace*{-0.15em} (0,1)$,\!
        if~$\bE_0$~is~sufficiently~smooth. This yields 
	that $\smash{p = \hat p \circ \abs{\bE}^2}$ satisfies Assumption
	\ref{VssEp(x)}, as easy~calculations~show.
\end{Bem}

Under these assumptions, we have the following weak stability~for~\mbox{problem~\!\eqref{NS}}.
\begin{Lem}\label{DalMaso3p(x)}
  Let $\Omega\subseteq\R^d$, $d\ge2$, be a bounded domain and let
  \mbox{Assumption~\ref{VssS}}, Assumption~\ref{VssN}~and
  Assumption~\ref{VssEp(x)} be satisfied. Moreover, let
  $\smash{(\bv^n)_{n\in \mathbb{N}}\!\subseteq\! V_{p(\cdot)}(\Omega)}$ and
  $\smash{(\w^n)_{n\in \mathbb{N}}\subseteq H^{1,p(\cdot)}_0(\Omega;\vert
  \bfE\vert^2)}$ be~such~that
	\begin{align*}
			\bv^n&\rightharpoonup\bv\quad& &\text{in}\
			V_{p(\cdot)}(\Omega)&&\quad(n\to \infty)\,,
			\\[-1mm]
			\w^n&\rightharpoonup\w\quad& &\text{in}\ H^{1,p(\cdot)}_0(\Omega;\vert \bfE\vert^2)&&\quad(n\to \infty) \,.
	\end{align*}      
	For every ball $B\subset\subset \Omega_0$ such that
        $B':=2B\subset \subset \Omega_0$ and $\tau\in C_0^\infty(B')$
        satisfying ${\chi_B\leq \tau\le\chi_{B'}}$, we set
        ${
          \mathbf{u}^n:=(\mathbf{v}^n-\mathbf{v})\tau,\bfpsi^n:=(\boldsymbol\omega^n-\boldsymbol\omega)\tau
          \in W^{1,p(\cdot)}_0(B')}$, $n\in \mathbb{N}$.  Let
        $ \mathbf{u}^{n,j}\in W^{1,\infty}_0(B')$,
        $n,j\in \mathbb{N}$, and $\bfpsi^{n,j}\in W^{1,\infty}_0(B')$,
        $n,j\in \mathbb{N}$, resp., denote the Lipschitz truncations
        constructed according to Theorem~\ref{thm:Ltp(x)}.~Furthermore,
        assume that for every ${j\in \mathbb{N}}$, we have that
	\begin{align*} 
		\limsup_{n\to\infty}\big\vert\big\langle&\Sn-\Ss,
		\D\bu^{n,j}+\Rr(\bu^{n,j},\bfpsi^{n,j})\big\rangle 
		\\[-1mm]
		&\quad + \big\langle\Nn-\N,\nabla\bfpsi^{n,j}\big\rangle
		\big \vert \le
		\delta_j\,, 
	\end{align*}
	where $\delta_j\to 0$ $(j\to 0)$. \!Then, \!one has $
	\nabla\bv^n\to\nabla\bv$ a.e.~in $B$~${(n\to
          \infty)}$,~${\nabla\w^n\to\nabla\w}$ a.e.~in $B$ $(n\to
        \infty)$ and $\w^n\to\w$ a.e.~in $B$ $(n\to \infty)$~for suitable~subsequences.
\end{Lem}

\begin{Bew}
  We follow, word by word, the procedure as in the proof of Lemma~\ref{DalMaso3}. In doing so, we employ Lemma \ref{hatgradp(x)} instead
  of Lemma \ref{hatgrad}, which results in
  $H^{1,p(\cdot)}(B';\abs{\bE}^2) = W^{1,p(\cdot)}(B')$. The trivial
  embedding $W^{1,p(\cdot)}(B')\vnor W^{1,p^-}(B')$ together with the
  classical Rellich's compactness theorem yields that we have to
  replace $q\in \left[1,p^*\right)$ by $q\in
  \left[1,(p^-)^*\right)$. Moreover, we have to replace the constant
  exponent $p\in \left(1,\infty\right)$ by the variable exponent
  $p\in \mathcal{P}^{\log}(\Omega)$, wherever it occurs. This applies,
  in particular, to all Lebesgue, weighted Lebesgue, Sobolev and
  weighted Sobolev norms containing $p$ or $p'$.
  Whenever we use H\"older's~inequality, \linebreak we get an additional multiplicative factor 2. Finally, we
  replace~$\smash{\|\bfE\|_{\infty}^{\smash{2/p}}\!}$~by~$\smash{\|\bfE\|_{\infty}^{\smash{2/p^-}}\!}$, $\abs{\Omega}^{\smash{1/p}}$ by
  $\max \{ \abs{\Omega}^{ \smash{1/p^+}}, \abs{\Omega}^{ \smash{1/p^-}}\big \}$ (cf.~\cite[Lem.~3.2.12]{lpx-book}) and
  $2^{-j/p}$ by $2^{-j/p^+}$.
\end{Bew}
\begin{Kor}\label{cor:ae-convpx}
  Let the assumptions of Lemma \ref{DalMaso3p(x)} be satisfied for all
  balls ${B\hspace*{-0.1em}\subset \subset\hspace*{-0.1em} \Omega_0}$ with
  ${B'\hspace*{-0.1em}:=\hspace*{-0.1em}2B\hspace*{-0.1em} \subset \subset\hspace*{-0.1em} \Omega_0}$.
  Then, one has that
  $\nabla \bv^n \hspace*{-0.1em}\to\hspace*{-0.1em} \nabla \bv$~a.e.~in~$\Omega$~${(n\hspace*{-0.1em}\to\hspace*{-0.1em} \infty)}$,
  $ \hat\nabla \w^n \to \hat\nabla \w$ a.e.~in $\Omega$
  $(n\to \infty)$ and $\w^n \to \,\w $~a.e.~in~$\Omega$~${(n\to \infty)}$~for~suitable~subsequences.
\end{Kor}
\begin{proof}
  The proof coincides with that of Corollary \ref{cor:ae-conv}. 
\end{proof}

\subsection{Existence theorem for variable exponents}\label{sec:mainp(x)}

\hspace*{-0.1cm}Now we have all tools at our disposal to formulate and prove~our~\mbox{existence~\!result} in the case of variable exponents.
\begin{Sa}\label{thm:main4p(x)}
  Let $\Omega\subseteq\R^d$, $d\ge2$, be a bounded domain, let
  Assumption~\ref{VssSpx}, Assumption~\ref{VssNpx} ~and Assumption
  ~\ref{VssEp(x)} be satisfied, and~let~${p^->\frac{2d}{d+2}}$.~Then, for~every
  $\smash{\ff\in (W^{1,p(\cdot)}_0(\Omega))^*}$ and
  $\smash{\bell\in (H^{1,p(\cdot)}_0(\Omega;\vert \bfE\vert^2))^*}$, there
  exist functions ${\mathbf{v}\hspace*{-0.1em}\in\hspace*{-0.1em} V_{p(\cdot)}(\Omega)}$ and
  $ {\w\hspace*{-0.1em}\in\hspace*{-0.1em} H^{1,p(\cdot)}_0(\Omega;\vert \bfE\vert^2)}$ such that for
  every $\boldsymbol\varphi \hspace*{-0.1em}\in\hspace*{-0.1em}
  C^1_0(\Omega)$ with ${\divo \bphi \hspace*{-0.1em}= \hspace*{-0.1em}0}$ and 
  $\boldsymbol\psi \in C_0^1(\Omega)$ with
  ${\nabla \boldsymbol\psi\in L^{\smash{\frac{q}{q-2}}}(\Omega;\vert
    \bfE\vert^{\smash{-\frac{\alpha q}{q-2}}})}$ for some
  ${q\in \left[1,(p^-)^*\right)}$, it holds
  \begin{align*}
      &\big\langle
      \Ss-\bv\otimes\bv,\bD\boldsymbol\varphi+\bR(\bphi,\boldsymbol\psi)\big\rangle
      \\
      & \;+\big\langle \bN(\hat \nabla \w
      ,\bE)-\w\otimes\bv,\nabla\boldsymbol\psi\big\rangle =\big\langle
      \ff,\boldsymbol\varphi\big\rangle+\big\langle\bell,
      \boldsymbol\psi\big\rangle \,.
  \end{align*}
  Moreover, we have the following a-priori estimate 
  \begin{equation*}
    \smash{\Vert\bv\Vert_{1,p(\cdot)} +\Vert\w\Vert_{1,p(\cdot),\vert\E\vert^2} \leq
      c\, \big (\norm{\bE}_2, \Vert
      \ff\Vert_{(\Wpo)^*}, \Vert\bell\Vert_{(\WpEo)^*}\big
      )\,.}
  \end{equation*}
\end{Sa}\newpage

\begin{Bew}
  We follow, word by word, the procedure  as in the proof of Theorem~\ref{thm:main4}. In doing so, we again have to replace the constant
  exponent $p\in \left(1,\infty\right)$ by the variable exponent
  $p\in \mathcal{P}^{\log}(\Omega)$, classical Lebesgue, weighted
  Lebesgue, Sobolev and weighted Sobolev norms containing $p$ or $p'$
  by~their~variable exponent counterparts. Moreover, we replace
  $r>2p'$ by $r>2(p^-)'$ in the definition of the approximate problem.
  To show that \eqref{apri0} implies \eqref{apri} in the variable~\mbox{exponent}~case,~the constant exponent Korn's and \Poincare's
  inequalities is replaced by their variable exponent counterparts in
  Theorem \ref{kornp(x)} and Theorem~\ref{poincarep(x)}, and
  \cite[Lem.~3.2.5.]{lpx-book} is used to pass from the modular
  estimate to the norm estimate.  Concerning the usage of Rellich's
  compactness theorem, we proceed as in the proof of
  Lemma~\ref{DalMaso3p(x)} and, thus, replace $q\in \left[1,p^*\right)$
  by $q\in \left[1,(p^-)^*\right)$. Moreover, we replace  Lemma
  \ref{hatgrad} by Lemma \ref{hatgradp(x)}, Theorem
  \ref{bog} by Theorem \ref{bogp(x)}, Theorem~\ref{thm:Lt} by Theorem
  \ref{thm:Ltp(x)} and Corollary \ref{cor:ae-conv} by Corollary
  \ref{cor:ae-convpx}. 
\end{Bew}

\section*{References}

{\linespread{0.7}\selectfont
\def\cprime{$'$} \def\cprime{$'$} \def\cprime{$'$}
\ifx\undefined\bysame
\newcommand{\bysame}{\leavevmode\hbox to3em{\hrulefill}\,}
\fi

\end{document}